\documentclass[11pt]{amsart}

\usepackage{amsfonts, amsmath, amssymb, amsthm, stmaryrd, color, enumerate, relsize}
\usepackage{url} 
\usepackage{here}
\usepackage{bbm}
\usepackage[all]{xy}
\usepackage{amscd}
\usepackage{graphicx}
\usepackage{mathrsfs,array}

\numberwithin{equation}{section}
\newtheorem{thm}{Theorem}
\numberwithin{thm}{section}
\newtheorem{lem}[thm]{Lemma}
\newtheorem{cor}[thm]{Corollary}
\newtheorem{prop}[thm]{Proposition}

\newtheorem{dfn}[thm]{Definition}

\theoremstyle{definition}
\newtheorem{rem}[thm]{Remark}

\newcommand{\RR}{\mathbb{R}}
\newcommand{\QQ}{\mathbb{Q}}
\newcommand{\ZZ}{\mathbb{Z}}
\newcommand{\CC}{\mathbb{C}}
\newcommand{\FF}{\mathbb{F}}

\newcommand{\bC}{\mathbb{C}}

\newcommand{\bF}{\mathbb{F}}

\newcommand{\bP}{\mathbb{P}}
\newcommand{\bQ}{\mathbb{Q}}

\newcommand{\bZ}{\mathbb{Z}}

\newcommand{\cO}{\mathcal{O}}
\newcommand{\cP}{\mathcal{P}}

\newcommand{\Hom}{\operatorname{Hom}}
\newcommand{\Ext}{\operatorname{Ext}}
\newcommand{\gal}{\operatorname{Gal}}
\newcommand{\gl}{\operatorname{GL}}

\newcommand{\pgl}{\operatorname{PGL}}
\newcommand{\ord}{\operatorname{ord}}
\DeclareMathOperator{\lcm}{lcm}
\newcommand{\SL}{\operatorname{SL}}
\newcommand{\GL}{\operatorname{GL}}
\newcommand{\PGL}{\operatorname{PGL}}
\newcommand{\Gal}{\operatorname{Gal}}

\title{Classifications of Hecke Fields and Galois Images of Weight One Exotic Newforms}
\author{Ryotaro Sakamoto and Sho Yoshikawa}

\address{R.S.: Department of Mathematics\\University of Tsukuba\\1-1-1 Tennodai\\Tsukuba\\Ibaraki 305-8571\\Japan}
\email{rsakamoto@math.tsukuba.ac.jp}

\address{S.Y.: Department of Mathematics\\Faculty of Science Division I\\Tokyo University of Science. 1-3 Kagurazaka\\Shinjuku-ku\\Tokyo 162-8601\\Japan}
\email{yoshikawa@rs.tus.ac.jp}

\date{\today}

\begin{document}

\begin{abstract}
We determine the Hecke fields associated with weight one newforms of $A_4$-, $S_4$-, and $A_5$-type, 
expressed in terms of the order of its nebentypus. 
Furthermore, for each type, we provide a complete classification of the images of the corresponding Galois representations. 
\end{abstract}
    
\maketitle
\setcounter{tocdepth}{2}
\tableofcontents

\section{Introduction}\label{intro}

Let 
\[f=\sum_{n\geq 1} a_n(f)q^n\in S_1(N,\chi)\]
be a weight one newform of level $N$ and nebentypus $\chi$ with $\chi(-1)=-1$.
By the theorem of Deligne and Serre \cite[Th\'eor\`eme 4.1]{DS}, there exists a continuous irreducible odd Galois representation
\[
\rho_f \colon G_\QQ := \operatorname{Gal(\overline{\QQ}/\QQ)} \to \gl_2(\CC)
\]
associated with $f$, in the sense that $\rho_f$ is unramified at  each prime $p \nmid N$ and the characteristic polynomial of $\rho_f(\mathrm{Frob}_p)$ is $X^2-a_p(f)X+\chi(p)$.

It is well-known that the projective image of $\rho_f$, namely the image of $\rho_f$ in $\mathrm{PGL}_2(\mathbb{C}) := \mathrm{GL}_2(\mathbb{C})/\mathbb{C}^\times$, is isomorphic to one of the finite groups $D_{2n}$, $A_4$, $S_4$, or $A_5$. 
If the projective image is isomorphic to the dihedral group $D_{2n}$ of order $2n$ for some integer $n$, then we say that $f$ is of dihedral type; otherwise, we say that $f$ is exotic, and we further distinguish the cases of $A_4$-type, $S_4$-type, or $A_5$-type, respectively.

Let $K_f$ be the Hecke field associated with $f$; that is, 
\[
K_f := \QQ(\{a_n(f)\}_{n\geq 1}).  
\]

\begin{rem}
The Galois representation $\rho_f$ is defined over $K_f$, that is, it is conjugate to a homomorphism $G_{\bQ} \to \GL_2(K_f)$.
Although this does not follow directly from the definition of $K_f$, it is a consequence of the fact that $\rho_f$ is odd (see, for instance, the discussion on MathOverflow \cite{MathOverflow_Field_of_Definition_of_Galois_repn}).
\end{rem}

Since $\rho_f$ has finite image, the eigenvalues of $\rho_f(\mathrm{Frob}_p)$ for $p \nmid N$ are roots of unity; hence the Frobenius trace $a_p(f)$ lies in a cyclotomic field. 
Consequently, $K_f$ is an abelian extension of $\mathbb{Q}$.
In the first part of this paper, we determine the Hecke field $K_f$ of a weight one exotic newform $f$ in terms of the order of its nebentypus $\chi$. 

In what follows, $d$ will always stand for the order of the nebentypus $\chi$;
\[
d := \mathrm{ord}(\chi). 
\]
We note that, as $\chi(-1) = -1$, the order $d$ is necessarily even. 
For any positive integer $n$, let $\zeta_n\in \CC$ denote a primitive $n$th root of unity. 
In this setting, the following theorems are among the main results of the present paper. 

\begin{thm}[Theorem \ref{thm:mainA4general}]
If $f$ is of $A_4$-type, then $K_f=\QQ(\zeta_{2d})$.
\end{thm}

\begin{thm}[Theorem \ref{thm:mainA5general}]
If $f$ is  of $A_5$-type, then $K_f=\QQ(\zeta_{2d}, \sqrt{5})$. 
\end{thm}

\begin{thm}[Theorem \ref{thm:mainS4general}]\label{intro_mainS4general}
Let $\mathrm{ord}_2(d) \geq 1$ denote the $2$-adic  valuation of the even integer $d$. If $f$ is of $S_4$-type, then 
\[
    K_f = 
\begin{cases}
\QQ(\zeta_{d},\sqrt{-2}) \,\, \textrm{ or } \,\, \QQ(\zeta_{4d}) & \textrm{ if } \, \mathrm{ord}_2(d) = 1, 
\\
\QQ(\zeta_{d}) \,\, \textrm{ or } \,\, \QQ(\zeta_{2d}) & \textrm{ if } \, \mathrm{ord}_2(d) = 2, 
\\
\QQ(\zeta_{2d})  & \textrm{ if } \, \mathrm{ord}_2(d) \geq 3.  
\end{cases}
\]
\end{thm}

\begin{rem}
\label{rem_twists}
For any exotic type and any even positive integer $d$, there exists a weight one  newform of the specified type whose nebentypus has order $d$.
This can be shown by twisting with an appropriate Dirichlet character, as follows: 

Let $g$ be a weight one exotic newform of the specified type whose nebentypus has order $2$ (the existence of such a newform $g$ can easily be verified, for example, by checking LMFDB \cite{LMFDB}). 
Take a prime $p$ such that $p \equiv 1 \pmod{2d}$ and $p$ is coprime to the level of $g$.
Then there exists a Dirichlet character $\psi$ of order $2d$ with conductor $p$, and we consider the weight one newform $f$ corresponding to $g \otimes \psi$, that is, $a_q(f) = \psi(q)a_q(g)$ for almost all primes $q$.  
Since $\rho_f \simeq \rho_g \otimes \psi$, the newform $f$ has the same type as $g$, and its nebentypus has order $d$, as follows from $\det(\rho_f) = \det(\rho_g)  \psi^2$. 
\end{rem}

\begin{rem}
       It follows from a result of Hecke and Strum that the Hecke algebra $\bZ[T_n \mid n \in \bZ_{>0}] \subset \mathrm{End}(S_1(N, \chi))$ 
   is generated as a $\bZ[\chi]$-algebra by $T_1$ and $T_p$ for all prime $p$ satisfying 
   \[
   p \leq \frac{1}{12}[\SL_2(\bZ) : \Gamma_0(N)]. 
   \]
   See \cite[Proposition 8.2.3]{Computing_classical_modular_forms_2011} for example. 
   Since the newform $f$ defines a surjective ring homomorphism 
   $\bQ[T_n \mid n \in \bZ_{>0}] \twoheadrightarrow  K_f; T_n \mapsto a_n(f)$, the Hecke field $K_f$ is generated as a $\bQ(\chi)$-algebra by $a_p(f)$ for all prime $p \leq [\SL_2(\bZ) : \Gamma_0(N)]/12$.  
\end{rem}


\begin{rem}
Throughout this remark, we assume that the level $N$ of $f$ is square-free.
Under this assumption, the possible orders of $\chi$ are highly restricted, and this enables us to obtain more explicit results, as follows (see Corollaries \ref{mainA4sqfree}, \ref{mainA5sqfree}, and \ref{mainS4sqfree}).
\begin{itemize}
    \item[(i)] If $f$ is of $A_4$-type, then $d = 6$ and $K_f=\QQ(\zeta_{12})$. 
    \item[(ii)] If $f$ is of $A_5$-type, then $d\in \{2,6,10,30\}$ and  
    \[
    K_f = 
    \begin{cases}
        \QQ(\zeta_{4}, \sqrt{5})  & \textrm{ if } \, d=2, 
        \\
        \QQ(\zeta_{12}, \sqrt{5}) & \textrm{ if } \, d=6, 
        \\
        \QQ(\zeta_{20}) & \textrm{ if } \, d=10, 
        \\
        \QQ(\zeta_{60}) & \textrm{ if } \, d=30. 
    \end{cases}
    \] 
    \item[(iii)] If $f$ is of $S_4$-type, then $d\in \{2,4,6,12\}$ and 
    \[
    K_f = 
    \begin{cases}
        \QQ(\sqrt{-2}) \, \textrm{ or } \, \QQ(\zeta_8)  
        & \textrm{ if } \, d=2, 
        \\
        \QQ(\zeta_4) \, \textrm{ or } \, \QQ(\zeta_8) 
        & \textrm{ if } \, d=4, 
        \\
        \QQ(\sqrt{-2}, \zeta_3) \, \textrm{ or } \, \QQ(\zeta_{24}) 
        & \textrm{ if } \, d=6, 
        \\
        \QQ(\zeta_{12}) \, \textrm{ or } \, \QQ(\zeta_{24}) 
        & \textrm{ if } \, d=12. 
    \end{cases}
    \] 
\end{itemize}
Moreover, a simple check of LMFDB \cite{LMFDB} shows that each list contains at least one explicit example.
\end{rem}

In Theorem \ref{intro_mainS4general}, when $\mathrm{ord}_2(d) \leq 2$, two distinct Hecke fields may occur.
A natural question is how one can distinguish between these two possibilities. 
We provide an answer to this question in the following sense.
The details are given in \S\ref{S4refpf}.

\begin{thm}[Theorem \ref{thm:S4_refinement}]
Let $\overline{\rho}_f \colon G_\bQ \to \PGL_2(\bC)$ denote the projective Galois representation associated with $f$. 
If $f$ is of $S_4$-type, i.e, $\overline{\rho}_f(G_\bQ) \simeq S_4$, then the following hold: 
\begin{itemize}
    \item[(i)] If $\chi^{d/2} \neq \operatorname{sgn}\circ \bar{\rho}_f$ as $G_\QQ$-representations, then 
        \[
    K_f = \begin{cases}
        \QQ(\zeta_{4d}) & \textrm{ if } \, \mathrm{ord}_2(d)=1, 
        \\
        \QQ(\zeta_{2d}) & \textrm{ if } \, \mathrm{ord}_2(d)=2.  
    \end{cases}
    \]
    \item[(ii)] If $\chi^{d/2} = \operatorname{sgn}\circ \bar{\rho}_f$ as $G_\QQ$-representations, then
    \[
    K_f = \begin{cases}
        \QQ(\zeta_{d},\sqrt{-2}) & \textrm{ if } \, \mathrm{ord}_2(d)=1, 
        \\
        \QQ(\zeta_{d}) & \textrm{ if } \, \mathrm{ord}_2(d)=2.  
    \end{cases}
    \]
    \end{itemize}
\end{thm}



\begin{rem}
In his forthcoming work, Yu Miyazawa narrows down the possible Hecke fields
for weight one newforms of dihedral type. 
Combining his results with ours, we conclude that for each positive integer $n$, only finitely many number fields of degree at most $n$ arise as Hecke fields associated with weight one newforms.
\end{rem}


As an application of the above results, for each type of the weight one exotic newform $f$, we can compute the Dirichlet density of the set $\mathcal{P}_f$ of primes defined by
\[
\mathcal{P}_f := \{ p \mid \gcd(p,N)=1 \text{ and } K_f = \QQ(a_p(f)) \}.
\]
In what follows, let $\varphi$ denote the Euler's totient function.

\begin{thm}[{Theorem \ref{thm:density_A_4_case}}]\label{intro:density_A_4_case}
Suppose that $f$ is of $A_4$-type. 
Then the Dirichlet density $d(\cP_f)$ of the set $\cP_f$ is given by 
\begin{align*}
    d(\cP_f) = 
    \begin{cases}
        \dfrac{3 \varphi(d)}{4d} & \textrm{if} \quad  \ker(\chi) \not\subset \ker(G_\bQ \stackrel{\overline{\rho}_f}{\twoheadrightarrow} A_4 \twoheadrightarrow A_4^{\rm ab}), 
        \\
        \\
        \dfrac{\varphi(d)}{d} &   \textrm{if} \quad \ker(\chi) \subset \ker(G_\bQ \stackrel{\overline{\rho}_f}{\twoheadrightarrow} A_4 \twoheadrightarrow A_4^{\rm ab}). 
    \end{cases}
\end{align*}
\end{thm}

\begin{thm}[{Theorem \ref{thm:density_A_5_case}}]\label{intro:density_A_5_case}
Suppose that $f$ is of $A_5$-type. 
Then the Dirichlet density $d(\cP_f)$ of the set $\cP_f$ is given by 
\begin{align*}
    d(\cP_f) = 
    \begin{cases}
        \dfrac{2 \varphi(d)}{5d} & \textrm{if} \quad  5 \nmid d, 
        \\
        \\
        \dfrac{3\varphi(d)}{4d} &  \textrm{if} \quad  5 \mid d. 
    \end{cases}
\end{align*}
\end{thm}

\begin{thm}[{Theorem \ref{thm:density_S_4_case}}]\label{intro:density_S_4_case}
Suppose that $f$ is of $S_4$-type. 
\begin{itemize}
    \item[(i)] If $\mathrm{sgn} \circ \overline{\rho}_f  \neq \chi^{d/2}$ and $\mathrm{ord}_2(d) = 1$, 
then $\cP_f = \emptyset$, i.e., the Hecke field $K_f$ cannot be generated, as a $\mathbb{Q}$-algebra, by a single eigenvalue $a_p(f)$. 
    \item[(ii)] If $\mathrm{sgn} \circ \overline{\rho}_f  \neq \chi^{d/2}$ and $\mathrm{ord}_2(d) \geq 2$, 
then  the Dirichlet density $d(\cP_f)$ of the set $\cP_f$ is given by 
\begin{align*}
    d(\cP_f) = 
    \begin{cases}
        \dfrac{3 \varphi(d)}{8d} & \textrm{if} \quad  \mathrm{ord}_2(d) = 2, 
        \\
        \\
        \dfrac{5\varphi(d)}{8d} &  \textrm{if} \quad  \mathrm{ord}_2(d) \geq 3. 
    \end{cases}
\end{align*}
    \item[(iii)] If $\mathrm{sgn} \circ \overline{\rho}_f  = \chi^{d/2}$, 
then  the Dirichlet density $d(\cP_f)$ of the set $\cP_f$ is given by 
\begin{align*}
    d(\cP_f) = 
    \begin{cases}
        \dfrac{\varphi(d)}{2d} & \textrm{if} \quad  \mathrm{ord}_2(d) \neq 2, 
        \\
        \\
        \dfrac{7\varphi(d)}{8d} &  \textrm{if} \quad  \mathrm{ord}_2(d) = 2. 
    \end{cases}
    \end{align*}
\end{itemize}
\end{thm}

\begin{rem}
By the same method as in the proofs of Theorems~\ref{intro:density_A_4_case}, 
\ref{intro:density_A_5_case}, and \ref{intro:density_S_4_case}, 
the Dirichlet density of primes $p$ satisfying
$K_f = \mathbb{Q}(\chi)(a_p(f))$ can also be computed
(see Remark~\ref{rem:density_Q(chi)(a_p(f))}). 
\end{rem}

\begin{rem}
    \label{rem_KSW_density1}
In \cite[Corollary~1.1]{KSW08}, Koo--Stein--Wiese proved that if a non-CM newform
$g \in S_k(N,\chi)$ of weight $k \geq 2$ has no nontrivial inner twists,
then the Hecke field of $g$ is generated, as a $\mathbb{Q}$-algebra, by a single Hecke eigenvalue
$a_p(g)$ for a set of primes $p$ of density one.
The case where $g$ admits a nontrivial inner twist is treated in \cite[Corollary~1.4]{KSW08}. 
Our results, Theorems~\ref{intro:density_A_4_case}, \ref{intro:density_A_5_case}, and \ref{intro:density_S_4_case},
can be viewed as a weight one counterpart of their theorem.
See also \cite{Choi-Taguchi16} and \cite{Kumar-Sahoo23} for related results. 
 \end{rem}

As noted in Remark \ref{rem_twists}, twisting by Dirichlet characters allows the order of the nebentypus of a newform to vary almost arbitrarily.
This motivates a restriction to twist-minimal newforms.
However, even among twist-minimal newforms within the same twist class, the orders of the nebentypus need not coincide.
Accordingly, in each twist class, we would like to single out a twist-minimal newform whose nebentypus has minimal order.
To make this notion explicit, we introduce the notion of a strongly minimal newform in \S\ref{sec:Strongly minimal newforms} (see Definition \ref{def:strongly_minimal}). 
Roughly speaking, a newform $f$ is strongly minimal if it is twist-minimal and its nebentypus $\chi$ is minimally ramified at every prime among all twist-minimal newforms in the same twist class.
In \S\ref{sec:Strongly minimal newforms}, we show that the order of the nebentypus of a strongly minimal newform is subject to restrictions.

\begin{thm}[{Theorem \ref{thm:minimal_order}}]\label{intro:order_of_nebentypus_of_strongly_minimal_newform}
        Suppose that the weight one exotic newform $f$ is strongly minimal.  
    \begin{itemize}
        \item[(i)] If $f$ is of $A_4$-type, then the order $d$ of the nebentypus  takes the form $3 \cdot 2^k$ with $k \geq 1$. 
        \item[(ii)] If $f$ is of $A_5$-type, then  the order $d$ of the nebentypus  takes the form $2^{k} \cdot 3^{\delta_3} \cdot 5^{\delta_5}$ with $k \geq 1$ and $\delta_3, \delta_5 \in \{0,1\}$ 
        \item[(iii)] If $f$ is of $S_4$-type, then the order $d$ of the nebentypus  takes the form $2^{k}$ or $3 \cdot 2^k$ with $k \geq 1$. 
    \end{itemize}
\end{thm}

Furthermore, we show that, when $f$ is of $A_4$-type or $S_4$-type, there exists a strongly minimal newform realizing each of the possible orders of the nebentypus described above.

\begin{thm}[{Propositions \ref{prop:exicentece_A4_minimal}, \ref{prop:exicentece_S4_minimal_2^k}, and \ref{prop:exicentece_S4_minimal_3*2^k}}]\label{intro:order_of_nebentypus_of_strongly_minimal_newform_existence_A_4_and_S_4}
For any positive integer $k$, there exist strongly minimal newforms with the following types and orders of nebentypus: 
\begin{itemize}
    \item[(i)] $A_4$-type with nebentypus of order $3 \cdot 2^k$,
    \item[(ii)] $S_4$-type with nebentypus of order $2^k$,
    \item[(iii)] $S_4$-type with nebentypus of order $3 \cdot 2^k$.
\end{itemize}
\end{thm}

\begin{rem}
We expect a similar statement to hold for weight one newforms of $A_5$-type; however, it seems difficult to establish this using the same arguments as in the proof of Theorem \ref{intro:order_of_nebentypus_of_strongly_minimal_newform_existence_A_4_and_S_4}. 
\end{rem}

Finally, in \S\ref{sec:group_structure_of_the_image_of_rho_f}, we completely determine the group structure of the Galois image $\rho_f(G_\bQ)$.

\begin{thm}[{Theorem \ref{thm:Galois_image_A_4_type} and Proposition \ref{prop:A4_criteriaon_K_neq_G}}]\label{intro:galis_image_classification_A_4}
Let $I$  denote the $2\times 2$ identity matrix. 
Suppose that $f$ is of $A_4$-type. 
\begin{itemize}
    \item[(i)] If $\ker(\chi) \not\subset \ker(G_\bQ \stackrel{\overline{\rho}_f}{\twoheadrightarrow} A_4 \twoheadrightarrow A_4^{\rm ab} )$, then the Galois image $\rho_f(G_\QQ)$ is isomorphic to the group 
    \[
    (\SL_2(\bF_3) \times \ZZ/2d\ZZ)/\langle (-I, d) \rangle.
    \]
    \item[(ii)] If $\ker(\chi) \subset \ker(G_\bQ \stackrel{\overline{\rho}_f}{\twoheadrightarrow} A_4 \twoheadrightarrow A_4^{\rm ab} ) $, then the integer $d$ is divisible by $3$, and the Galois image $\rho_f(G_\QQ)$ is isomorphic to the group 
    \[
    (\SL_2(\bF_3) \times_{\ZZ/3\ZZ} \ZZ/2d\ZZ)/\langle (-I, d) \rangle.
    \]
Here, the homomorphism $\SL_2(\FF_3) \twoheadrightarrow \ZZ/3\ZZ$ used in the fiber product is defined as the composition
\[
\SL_2(\FF_3) \twoheadrightarrow \SL_2(\FF_3)^{\mathrm{ab}}  \simeq \ZZ/3\ZZ.
\]
\end{itemize}
\end{thm}

\begin{thm}[{Theorem \ref{thm:Galois_image_A_5_type}}]\label{intro:galis_image_classification_A_5}
Let $I$  denote the $2\times 2$ identity matrix. 
If $f$ is of $A_5$-type, then the Galois image $\rho_f(G_\QQ)$  is isomorphic to the group 
    \[
    (\SL_2(\bF_5) \times \ZZ/2d\ZZ)/\langle (-I, d) \rangle.
    \]
\end{thm}

\begin{thm}[{Theorem \ref{thm:Galois_image_S_4_type} and Proposition \ref{prop:S4_criteriaon_K_neq_G}}]\label{intro:galis_image_classification_S_4}
Let $BO_{48}$ be the binary dihedral group of order $48$ and 
$z \in BO_{48}$ denotes the unique element of order $2$ (see Remark \ref{rem:double_cover_S_4}). 
Suppose that $f$ is of $S_4$-type. 
\begin{itemize}
    \item[(i)] If $\chi^{d/2} \neq \operatorname{sgn}\circ \bar{\rho}_f$, then the Galois image $\rho_f(G_\QQ)$  is isomorphic to the group 
    \[
    (BO_{48} \times \ZZ/2d\ZZ)/\langle (z, d) \rangle. 
    \]
    \item[(ii)] If $\chi^{d/2} = \operatorname{sgn}\circ \bar{\rho}_f$, then the Galois image $\rho_f(G_\QQ)$  is isomorphic to the group 
    \[
(BO_{48} \times_{\bZ/2\bZ} \ZZ/2d\ZZ)/\langle (z, d) \rangle.
    \]
    Here, the homomorphism $BO_{48} \twoheadrightarrow \ZZ/2\ZZ$ used in the fiber product is defined as the composition
\[
BO_{48} \twoheadrightarrow BO_{48}^{\rm ab} \simeq  \ZZ/2\ZZ.
\]
\end{itemize}
\end{thm}

\subsection*{Acknowledgement}
The second author would like to thank Takeshi Ogasawara and Koji Shimuzu for helpful suggestions. 
RS was supported by JSPS KAKENHI Grant Number 24K16886.

\section{Preliminaries}\label{prelim}

\subsection{Projective Galois representation}

\begin{lem}[{\cite[Lemma 1]{BL}}]\label{lem:BL_c(g)}
For any matrix $g \in \gl_2(\bC)$, if the image  of $g$ in $\pgl_2(\bC)$ has order $n$, then 
\[
\frac{\mathrm{trace}(g)^2}{\det(g)} =  2+\zeta_n+\zeta_n^{-1}
\]
  for some primitive $n$-th root  of unity $\zeta_n$. 
\end{lem}
\begin{proof}
Since $\bar g$ has order $n$, the ratio of the two eigenvalues of the matrix $g$ is a primitive $n$-th root of unity. This fact immediately implies the lemma. 
\end{proof}




For each prime $p$, let $I_p \subset G_{\QQ_p}$ denote the  inertia subgroup at $p$. 

\begin{lem}\label{lem:order_character}
The order of a continuous character  $\mu \colon G_{\QQ} \to \CC^{\times}$ is equal to the least common multiple of the positive integers $|\mu(I_{p})|$ for all prime $p$. 
\end{lem}
\begin{proof}
This lemma is a consequence of the fact that the Galois group of a cyclotomic extension over $\mathbb{Q}$ is generated by its inertia groups. 
\end{proof}

 For any complex 2-dimensional Galois representation $\rho \colon G_\QQ \to \gl_2(\CC)$,  we write $\bar{\rho} \colon G_\QQ\to \pgl_2(\CC)$ for the projective Galois representation attached to $\rho$. 
 
The following result was proved by Nakazato in \cite{NH80}.
However, for the sake of completeness of the present paper, we include a proof.
Our proof of the following lemma is based on that of \cite[Theorem~7]{Se}, which treats the case of a prime conductor.

\begin{lem}[{\cite[Lemma 1]{NH80}}]\label{Serre}
Let $V$ be a $2$-dimensional $\bC$-vector space, and 
let $\rho \colon G_\QQ\to \gl(V)$ be a continuous irreducible odd Galois representation with square-free conductor $N$. Set $\chi := \det \rho$.
    Then the following hold.
    \begin{itemize}
        \item[(i)] If $\bar{\rho}(G_\QQ)$ is isomorphic to $A_4$, then $\chi$ has order $6$. 
        \item[(ii)] If $\bar{\rho}(G_\QQ)$ is isomorphic to $S_4$, then $\chi$ has order $2$, $4$, $6$, or $12$
        \item[(iii)] If $\bar{\rho}(G_\QQ)$ is isomorphic to $A_5$, then $\chi$ has order $2$, $6$, $10$, or $30$.
    \end{itemize}
\end{lem}
\begin{proof}
    Let $p$ be a prime dividing $N$. Since $N$ is square-free, we have $\dim_\CC V^{I_p} = 1$.
    This in turn implies that $V|_{I_p}\simeq 1\oplus \chi|_{I_p}$, and hence  $\bar{\rho}(I_p)\simeq \rho(I_p)\simeq \chi(I_p)$. 
    Therefore, Lemma \ref{lem:order_character} shows that 
    \[
    \mathrm{ord}(\chi) = \lcm\{|\chi(I_p)|\}_{p\mid N} =  \lcm\{|\bar{\rho}(I_p)|\}_{p\mid N}. 
    \]
    Note that, in this case, the group $\bar{\rho}(I_p)$ is cyclic for any prime $p$. 
    \begin{itemize}
        \item[(i)] Suppose that $\bar{\rho}(G_\QQ)$ is isomorphic to $A_4$. 
        Since $\bar\rho(I_p)$ injects into $A_4$, 
        the group $\bar\rho(I_p)$ is cyclic of order $1$, $2$, or $3$. Also, because $\rho$ is odd, we have $\mathrm{ord}(\chi) \in \{2,6\}$. We claim that $\mathrm{ord}(\chi) = 6$. Since $\bar{\rho}(G_\QQ)$ is isomorphic to $A_4$, the fixed field of $\ker \bar\rho$ contains a cyclic cubic extension $K/\QQ$ corresponding to the Klein four-group in $A_4$. Let $q$ be a prime ramified in $K$. Then $q \mid N$ and $\bar{\rho}(I_q)$ surjects onto $\operatorname{Gal}(K/\QQ)$. Thus $3$ divides $|\bar{\rho}(I_q)|$ and so the order of $\chi$ is $6$.  
        
        \item[(ii)] Suppose that $\bar{\rho}(G_\QQ)$ is isomorphic to $S_4$. Since  $\bar{\rho}(I_p)$ injects into $S_4$, the group $\bar{\rho}(I_p)$ is cyclic of order $1$, $2$, $3$, or $4$. Combined with the fact that $\chi$ has even order, this implies that $\mathrm{ord}(\chi)$ is $2$, $4$, $6$, or $12$.
        
        \item[(iii)] Suppose that $\bar{\rho}(G_\QQ)$ is isomorphic to $A_5$. Since $\bar{\rho}(I_p)$ injects into $A_5$, the group $\bar{\rho}(I_p)$ is cyclic of order $1$, $2$, $3$, or $5$. 
        Since $\chi$ has even order, it follows that $\mathrm{ord}(\chi)$ is  $2$, $6$, $10$, or $30$.
    \end{itemize}
\end{proof}

\subsection{Density of $Q_m(\chi)$} 


 For any positive integer $N$, using the canonical isomorphism
    \[
 (\ZZ/N\ZZ)^\times \stackrel{\sim}{\to} \operatorname{Gal}(\QQ(\zeta_N)/\QQ); \, p  \mapsto (\mathrm{Frob}_p \colon \zeta_N \mapsto \zeta_N^p), 
 \]
we identify, as usual, a Dirichlet character $\chi \colon \ZZ/N\ZZ \to \CC$ of conductor dividing $N$ with a Galois character $\chi \colon  \operatorname{Gal}(\QQ(\zeta_N)/\QQ)\to \CC^\times$. 

\begin{dfn}
For any integer $N$, let $\cP_N$ denote the set of all primes that are coprime to $N$.
For any positive integer $m$ and any character $\chi \colon \operatorname{Gal}(\QQ(\zeta_N)/\QQ)\to \CC^\times$,  we define a set $Q_m(\chi)$ of primes by 
    \[
    Q_m(\chi) := \left\{p \in \cP_N  \mid 
    \textrm{$\ord(\chi(p))$ is divisible by $m$} 
    \right\}. 
    \]
Throughout, when $\chi$ is clear from context, we denote $Q_m$ simply by omitting $\chi$.
\end{dfn}

\begin{lem}
\label{density}
    Let $\chi \colon \gal(\QQ(\zeta_N)/\QQ) \to \CC^\times$ be a character of order $d$, and take  a prime divisor  $\ell$ of $d$.
    Write $d=\ell^ed'$ with $(\ell,d')=1$. Then, the Dirichlet density $d(Q_{\ell^e}(\chi))$ of $Q_{\ell^e}(\chi)$ is $1-\frac{1}{\ell}$. 
\end{lem}
\begin{proof}
    Inside $\ZZ/d\ZZ\simeq \ZZ/\ell^e\ZZ\times \ZZ/d'\ZZ$, there are $(\ell^e-\ell^{e-1})d'$ elements of order divisible by $\ell^e$.
    By applying the Chebotarev density theorem to the fixed field of $\ker(\chi)$, 
    we obtain  
    \[d(Q_{\ell^e}(\chi))=\frac{(\ell^e-\ell^{e-1})d'}{d}=1-\frac{1}{\ell}\]
    as desired.
\end{proof}

\section{Hecke field $K_f$}
\label{sec:hecke}

Let $f =\sum_{n\geq 1} a_n(f) q^n\in S_1(N,\chi)$ be a weight one newform of level $N$ and nebentypus $\chi$ with $\chi(-1)=-1$,
and 
\[
\rho_f \colon G_\QQ \to \gl_2(\CC)
\]
denotes the continuous irreducible odd Galois representation associated with $f$. 
Recall that we write $d$ for the order of $\chi$, which is even  since $\chi(-1)=-1$.

\begin{dfn}
For any positive integer $m$, we define a set $R_m$ of primes  (depending on $f$)  by 
        \[ 
        R_m := \left\{p \in \cP_N \mid  \textrm{$\bar{\rho}_f(\mathrm{Frob}_p)$ has order $m$}   \right\}. 
        \]
\end{dfn}

\begin{lem}\label{lem:hecke_field_ap}
    For any prime $p \in \cP_N$, we have 
\[
        \QQ(a_p(f)) = \begin{cases}
\QQ( \sqrt{\chi(p)}) & \text{ if } \, p \in R_1 \sqcup R_3,\\
\QQ & \text{ if } \, p \in R_2,\\
\QQ( \sqrt{2\chi(p)}) & \text{ if } \, p \in R_4, \\
\QQ(\sqrt{5}, \sqrt{\chi(p)}) & \text{ if } \, p \in R_5. 
\end{cases}
\]
\end{lem}
\begin{proof}
Lemma \ref{lem:BL_c(g)}, together with the properties of the Galois representation $\rho_f$, shows that
\[
\frac{a_p(f)^2}{\chi(p)} = 
\begin{cases}
4 & \text{ if } \, p \in R_1,\\
0 & \text{ if } \, p \in R_2,\\
1 & \text{ if } \, p \in R_3, \\
2 & \text{ if } \, p \in R_4,\\
\dfrac{3 \pm \sqrt{5}}{2} & \text{ if } \, p \in R_5. 
\end{cases}
\]
When $p \in R_1 \sqcup R_2 \sqcup R_3 \sqcup R_4$, this lemma is an immediate consequence of this formula. 
Let us consider the case where $p \in R_5$. 
Since $\sqrt{\frac{3\pm \sqrt{5}}{2}}=\frac{\sqrt{5}\pm 1}{2}$, we obtain 
\[
\QQ(a_p(f)) = \QQ((\sqrt{5}\pm 1)\sqrt{\chi(p)} ). 
\]
Moreover,   
\[
(\sqrt{5} \pm 1)^{2d} = \left((\sqrt{5} \pm 1)\sqrt{\chi(p)} \right)^{2d} \in \QQ(a_p(f)), 
\]
so that $\sqrt{5} \in \QQ(a_p(f))$. 
It then follows that  
\[
\sqrt{\chi(p)} = \frac{1}{4}\left(\sqrt{5}(\sqrt{5} \pm 1)\sqrt{\chi(p)} \mp (\sqrt{5} \pm 1)\sqrt{\chi(p)}\right) \in \QQ(a_p(f)), 
\]
and hence $\QQ(a_p(f)) = \QQ(\sqrt{5}, \sqrt{\chi(p)})$. 
\end{proof}

\subsection{$A_4$-case}\label{mainA4pf}


In this subsection, we determine the Hecke fields of weight one newforms of $A_4$-type.

\begin{thm}\label{thm:mainA4general} 
If $f$ is of $A_4$-type, then $K_f=\QQ(\zeta_{2d})$.
\end{thm}  
\begin{proof}
 Since $f$ is of $A_4$-type, 
 for any prime $p \in \cP_N$, the element  $\bar{\rho}_f(\mathrm{Frob}_p)\in \pgl_2(\CC)$  has order $1$, $2$, or $3$. 
 Thus we have $\cP_N = R_1\sqcup R_2\sqcup R_3$. 
As $d$ denotes the order of $\chi$, it follows from Lemma \ref{lem:hecke_field_ap} that $K_f \subset \QQ(\zeta_{2d})$.



We now prove the converse inclusion. 
Applying the Chebotarev density theorem to $M/\QQ$, where $M$ is the fixed field of $\ker \bar\rho_f$, we obtain 
 \[
 d(R_1)=\frac{1}{12}, \quad d(R_2)=\frac{3}{12}, \quad d(R_3)=\frac{8}{12},
 \]
 where $d(R_m)$ denotes the Dirichlet density of $R_m$.  
 Let $\ell$ be any prime factor of $d$ and denote by $e := \mathrm{ord}_\ell(d)$  the $\ell$-adic order of $d$. 
 By Lemma \ref{density}, the set $Q_{\ell^e} = Q_{\ell^e}(\chi)$ of primes has  density $1-\frac{1}{\ell}$,
 which is greater than $d(R_2)=\frac{3}{12}$.
 Thus $Q_{\ell^e}\not\subset R_2$, and so $Q_{\ell^e}\cap (R_1\sqcup R_3)\neq \emptyset$. Considering $a_p(f)$ for any prime $p\in Q_{\ell^e}\cap (R_1\sqcup R_3)$, we have, from Lemma \ref{lem:hecke_field_ap}, 
 \[
 \QQ(a_p(f)) = \QQ(\sqrt{\chi(p)})
 \]
 with $\ell^e \mid \ord(\chi(p))$. In particular, $\QQ(\zeta_{2\ell^e})\subset K_f$. Since $\ell$ is an arbitrary prime factor of $d$, it follows that $\QQ(\zeta_{2d})\subset K_f$. 
\end{proof}

\begin{cor}\label{mainA4sqfree}
Suppose that $f$ is of $A_4$-type and the level $N$ of $f$ is  \textit{square-free}. 
 Then $d = 6$ and  $K_f=\QQ(\zeta_{12})$. 
\end{cor}
\begin{proof}
    By \cite[Th\'eor\`eme 4.1]{DS}, the Galois representation $\rho_f$ has conductor $N$, and so Lemma \ref{Serre} implies that $\chi$ has order $6$. Therefore, the assertion follows from Theorem \ref{thm:mainA4general}.
\end{proof} 
%

\subsection{$A_5$-case}\label{thm:mainA5pf}

In this subsection, we determine the Hecke fields of weight one newforms of $A_5$-type.

\begin{thm}\label{thm:mainA5general}
If $f$ is of $A_5$-type, then $K_f=\QQ(\zeta_{2d}, \sqrt{5})$.
\end{thm}
\begin{proof}
Since $f$ is of $A_5$-type, the element $\bar{\rho}_f(\mathrm{Frob}_p)\in \pgl_2(\CC)$ has order $1$, $2$, $3$, or $5$. 
Hence we have $\mathcal{P}_N = R_1\sqcup R_2\sqcup R_3\sqcup R_5$, and 
Lemma \ref{lem:hecke_field_ap} shows that $K_f \subset \QQ(\zeta_{2d}, \sqrt{5})$. 

Let us  prove the converse inclusion. 
The Chebotarev density theorem implies that 
\[
d(R_1)=\frac{1}{60}, \quad d(R_2)=\frac{15}{60}, \quad  d(R_3)=\frac{20}{60}, \quad   d(R_5)=\frac{24}{60}.
\]
In particular, $R_5 \neq \emptyset$, and we have $\sqrt{5} \in K_f$ by Lemma \ref{lem:hecke_field_ap}.
Let $\ell$ be any prime factor of $d$ and denote by $e := \mathrm{ord}_\ell(d)$  the $\ell$-adic order of $d$. 
 By Lemma \ref{density}, the set $Q_{\ell^e}$ of primes has  density $1-\frac{1}{\ell}$,
 which is greater than $d(R_2)=\frac{15}{60}$.
 Thus  $Q_{\ell^e}\cap (R_1\sqcup R_3 \sqcup R_5)\neq \emptyset$. Considering $a_p$ for any prime $p\in Q_{\ell^e}\cap (R_1\sqcup R_3 \sqcup R_5)$, we have, from Lemma \ref{lem:hecke_field_ap}, 
 \[
\zeta_{2\ell^e} \in  \QQ(a_p(f)). 
 \]
 Since $\ell$ is an arbitrary prime factor of $d$, it follows that $\zeta_{2d} \in K_f$. Therefore, we conclude that $K_f = \QQ( \zeta_{2d}, \sqrt{5})$. 
\end{proof}

\begin{cor}\label{mainA5sqfree}
Suppose that $f$ is of $A_5$-type and the level $N$ of $f$ is  \textit{square-free}. 
    Then, $d\in \{2,6,10,30\}$ and  
    \[
    K_f = 
    \begin{cases}
        \QQ(\zeta_{4}, \sqrt{5})  & \textrm{ if } \, d=2, 
        \\
        \QQ(\zeta_{12} \sqrt{5}) & \textrm{ if }  \, d=6, 
        \\
        \QQ(\zeta_{20}) & \textrm{ if } \, d=10, 
        \\
        \QQ(\zeta_{60}) & \textrm{ if } \, d=30. 
    \end{cases}
    \] 
\end{cor}
\begin{proof}
    By \cite[Th\'eor\`eme 4.1]{DS}, the Galois representation $\rho_f$ has conductor $N$, and so Lemma \ref{Serre} implies that $d\in \{2,6,10,30\}$. Therefore, this corollary  follows immediately from Theorem \ref{thm:mainA5general}. Here, note that $\sqrt{5}\in \QQ(\zeta_5)$. 
\end{proof}

\subsection{$S_4$-case}\label{thm:mainS4pf}

In this subsection, we classify the Hecke fields of weight one newforms of $S_4$-type. 
The idea remains the same as in the cases corresponding to $A_4$ and $A_5$; however, the situation becomes more complicated due to the existence of order-$4$ elements in the projective image of the Galois representation $\rho_f$.

\begin{thm}\label{thm:mainS4general}
Let $k := \mathrm{ord}_2(d) \geq 1$ denote the $2$-adic valuation of the even integer $d$. If $f$ is of $S_4$-type, then the following hold: 
\begin{itemize}
    \item[(i)] If $k=1$, then $K_f=\QQ(\zeta_{d},\sqrt{-2})$ or $\QQ(\zeta_{4d})$.
    \item[(ii)] If $k=2$, then $K_f=\QQ(\zeta_{d})$ or $\QQ(\zeta_{2d})$.
    \item[(iii)] If $k\geq 3$, then $K_f=\QQ(\zeta_{2d})$.
\end{itemize}
\end{thm}
\begin{proof}
 Since $f$ is of $S_4$-type,
 the element  $\bar{\rho}_f(\mathrm{Frob}_p)\in \pgl_2(\CC)$ has order $1$, $2$, $3$, or $4$. Thus $\mathcal{P}_N = R_1\sqcup R_2 \sqcup R_3\sqcup R_4$, and  
the Chebotarev density theorem implies that the Dirichlet density of $R_m$ for each $m\in \{1,2,3,4\}$ is given by
\[
d(R_1)=\frac{1}{24}, \quad d(R_2)=\frac{9}{24}, \quad d(R_3)=\frac{8}{24}, \quad d(R_4)=\frac{6}{24}.
\]
Let $d' := d/2^k \in \ZZ$. 
Then, it follows a priori from Lemma \ref{lem:hecke_field_ap} that
\[
K_f \subset \QQ(\zeta_{2d}, \sqrt{2}) = \QQ(\zeta_{2^{k+1}}, \zeta_{d'}, \sqrt{2}).
\]

First, let us show $\zeta_{d'} \in K_f$. We may assume that $d' > 1$ since the case $d'=1$ is clear. 
Take  any odd prime divisor $\ell$ of $d'$ and denote by $e := \mathrm{ord}_\ell(d')$ the $\ell$-adic order of $d'$. By Lemma \ref{density}, the set $Q_{\ell^{e}}$ of primes has the density $1-\frac{1}{\ell}$, which is greater than $\frac{15}{24} = d(R_2\sqcup R_4)$. Hence, $Q_{\ell^{e}} \not\subset R_2\sqcup R_4$, and consequently $Q_{\ell^{e}}\cap (R_1\sqcup R_3)\neq \emptyset$. Taking $a_p(f)$ for $p\in Q_{\ell^{e}}\cap (R_1\sqcup R_3)$, we have, from Lemma \ref{lem:hecke_field_ap}, $\QQ(a_p(f)) = \QQ( \sqrt{\chi(p)})$ with $\ell^e \mid \ord(\chi(p))$. 
Since $\ell$ is an arbitrary odd prime factor of $d'$, it follows that $\zeta_{d'} \in K_f$. 

We now consider the set $Q_{2^k}$. Since $d(Q_{2^k})=\frac{1}{2}>d(R_2)=\frac{9}{24}$ by Lemma \ref{density},  we have $Q_{2^k}\cap (R_1\sqcup R_3\sqcup R_4)\neq \emptyset$. 
Lemma \ref{lem:hecke_field_ap} shows that 
\[
K_f \ni
\begin{cases}
\zeta_{2^{k+1}}  & \textrm{ if } \, Q_{2^k} \cap  (R_1\sqcup R_3) \neq \emptyset, 
\\
    \sqrt{2}\zeta_{2^{k+1}}  &  \textrm{ if } \, Q_{2^k} \cap R_4 \neq \emptyset, 
\end{cases}
\]
by considering $a_p(f)^{d'}$ for at least one prime $p$ in $Q_{2^k} \cap  (R_1\sqcup R_3)$ or $Q_{2^k} \cap R_4$.
\begin{itemize}
    \item[(i)] Suppose that $k=1$. 
     \begin{itemize}
        \item[(a)] When 
        $Q_2 \cap (R_1 \sqcup R_3) \neq \emptyset$, we have $\zeta_4 \in K_f$.
        If $Q_2\cap R_4\neq \emptyset$, then  $\sqrt{-2}\in K_f$, and hence $\sqrt{2}\in K_f$ since $\zeta_4\in K_f$.
        If $Q_2\cap R_4= \emptyset$, then primes $p\in (\mathcal{P}_N\setminus Q_2)\cap R_4$ yield $\sqrt{2\chi(p)}$ with $\chi(p)$ odd order, and hence $\sqrt{2}\in K_f$.
         In any case, we have  $K_f=\QQ(\zeta_4,  \zeta_{d'}, \sqrt{2})$, which is equal to $\QQ(\zeta_{4d})$ since $d=2d'$ and $\zeta_8=\pm \frac{1 \pm \sqrt{-1}}{\sqrt{2}}$. 
        
        \item[(b)] When $Q_2\cap (R_1\sqcup R_3) = \emptyset$,  we have $Q_2 \subset R_2 \sqcup R_4$  and so $\sqrt{-2}  \in K_f$. Hence, $\QQ(\zeta_{d'},\sqrt{-2})\subset K_f$. 
        Since $[\QQ(\zeta_4, \zeta_{d'},\sqrt{2}) \colon \QQ(\zeta_{d'},\sqrt{-2})]=2$, we have $K_f=\QQ(\zeta_{d},\sqrt{-2})$ or $\QQ(\zeta_4, \zeta_{d'},\sqrt{2}) =\QQ(\zeta_{4d})$. Here, we note $\zeta_{d} = -\zeta_{d'}$ since $d=2d'$ with $d'$ odd.  
        
     \end{itemize}
    \item[(ii)] Suppose that $k=2$. Note that $\QQ( \zeta_{2d}, \sqrt{2})=\QQ( \zeta_{2d})$ since $\sqrt{2}=\pm(\zeta_8+\zeta_8^{-1})\in \QQ(\zeta_8)\subset \QQ(\zeta_{2d})$.
    \begin{itemize}
        \item[(a)] When $Q_4 \cap (R_1 \sqcup R_3) \neq \emptyset$,  we have $\QQ(\zeta_8) \subset K_f$, and hence $\QQ(\zeta_{2d}) \subset K_f$, which must in fact be an equality.

        \item[(b)] When $Q_4 \cap (R_1 \sqcup R_3) = \emptyset$, we have $Q_4\cap R_4\neq \emptyset$, and hence  $\QQ(\zeta_4) = \QQ(\sqrt2\zeta_8) \subset K_f$. In this case, $\QQ(\zeta_d) = \QQ(\zeta_4, \zeta_{d'}) \subset K_f$, and so $K_f= \QQ(\zeta_d)$ or $\QQ(\zeta_{2d})$.
    \end{itemize} 
    \item[(iii)] Suppose that $k\geq 3$. 
    In this case, we observe that $2\zeta_{2^k} = (\sqrt{2}\,\zeta_{2^{k+1}})^2 \in K_f$ and that $\sqrt{2} \in \QQ(\zeta_8) \subset \QQ(\zeta_{2^k})$. 
It then follows from $Q_{2^k} \cap (R_1 \sqcup R_3 \sqcup R_4) \neq \emptyset$ that $\zeta_{2^{k+1}} \in K_f$, and therefore $K_f = \QQ(\zeta_{2^{k+1}}, \zeta_{d'}) = \QQ(\zeta_{2d})$. 
\end{itemize}
\end{proof}

\begin{cor}\label{mainS4sqfree}
Suppose that $f$ is of $S_4$-type and the level $N$ of $f$ is  \textit{square-free}. 
Then $d\in \{2,4,6, 12\}$ and 
    \[
    K_f = 
    \begin{cases}
        \QQ(\sqrt{-2}) \, \textrm{ or } \,  \QQ(\zeta_8)  & \textrm{ if } \, d=2, 
        \\
        \QQ(\zeta_4)\, \textrm{ or }\, \QQ(\zeta_8) & \textrm{ if }  \, d=4, 
        \\
        \QQ(\sqrt{-2}, \zeta_3)\, \textrm{ or }\, \QQ(\zeta_{24}) & \textrm{ if } \, d=6, 
        \\
        \QQ(\zeta_{12}) \, \textrm{ or } \, \QQ(\zeta_{24}) & \textrm{ if } \, d=12. 
    \end{cases}
    \] 
\end{cor}
\begin{proof}
     By \cite[Th\'eor\`eme 4.1]{DS} and Lemma \ref{Serre}, we have $d \in \{2, 4, 6, 12\}$. Therefore, this corollary follows immediately from Theorem \ref{thm:mainS4general}. 
\end{proof}

\subsection{Refinement of the $S_4$-case}
\label{S4refpf}

Throughout this subsection, we assume that the newform $f$ is of $S_4$-type. 
Since in this case the projective image $\bar{\rho}_f(G_{\bQ})$ is isomorphic to $S_4$,
we obtain a quadratic character
\[
G_{\bQ} \twoheadrightarrow \bar{\rho}(G_{\bQ}) \twoheadrightarrow
\bar{\rho}(G_{\bQ})/[\bar{\rho}(G_{\bQ}), \bar{\rho}(G_{\bQ})]
\simeq \{\pm 1\}.
\]
We denote this quadratic character by $\operatorname{sgn} \circ \bar{\rho}$.


\begin{thm}\label{thm:S4_refinement}
Suppose that $f$ is of $S_4$-type. 
\begin{itemize}
    \item[(i)] If $\chi^{d/2} \neq \operatorname{sgn}\circ \bar{\rho}$ as $G_\QQ$-representations, then 
        \[
    K_f = \begin{cases}
        \QQ(\zeta_{4d}) & \textrm{ if } \, \mathrm{ord}_2(d)=1, 
        \\
        \QQ(\zeta_{2d}) & \textrm{ if } \, \mathrm{ord}_2(d)=2.  
    \end{cases}
    \]
    \item[(ii)] If $\chi^{d/2} = \operatorname{sgn}\circ \bar{\rho}$ as $G_\QQ$-representations, then
    \[
    K_f = \begin{cases}
        \QQ(\zeta_{d},\sqrt{-2}) & \textrm{ if } \, \mathrm{ord}_2(d)=1, 
        \\
        \QQ(\zeta_{d}) & \textrm{ if } \, \mathrm{ord}_2(d)=2.  
    \end{cases}
    \]
\end{itemize}
\end{thm}

The proof of Theorem \ref{thm:S4_refinement}(i) is given in \S\ref{sec:proof_of_S_4_refinement_i}, and that of Theorem \ref{thm:S4_refinement}(ii) is given in \S\ref{sec:proof_of_S_4_refinement_ii}.

\begin{cor}
\label{S4refcor}
Suppose that $f$ is of $S_4$-type and the level $N$ of $f$ is  \textit{square-free}. 
    Then, $d\in \{2,4,6,12\}$ and the following hold.   
\begin{itemize}
    \item[(i)] If $\chi^{d/2} \neq \operatorname{sgn}\circ \bar{\rho}$ as $G_\QQ$-representations, then 
        \[
    K_f = \begin{cases}
        \QQ(\zeta_{8}) & \textrm{ if } \, d=2, 4, 
        \\
        \QQ(\zeta_{24}) & \textrm{ if } \, d= 6, 12. 
    \end{cases}
    \]
    \item[(ii)] If $\chi^{d/2} = \operatorname{sgn}\circ \bar{\rho}$ as $G_\QQ$-representations, then
    \[
    K_f = \begin{cases}
\QQ(\sqrt{-2})   & \textrm{ if } \, d = 2, 
\\
\QQ(\zeta_4)  & \textrm{ if } \, d = 4, 
\\
\QQ(\zeta_3, \sqrt{-2}) & \textrm{ if } \, d=6, 
\\
\QQ(\zeta_{12})  & \textrm{ if } \, d=12.  
    \end{cases}
    \]
\end{itemize}
\end{cor} 
\begin{proof}
    This is an immediate consequence of  Lemma \ref{Serre} and  Theorem \ref{thm:S4_refinement}.
\end{proof}



\begin{rem}
    Consider here the case of prime conductor. Let $\rho: G_\QQ\to \gl_2(\CC)$ be a continuous irreducible  $2$-dimensional Galois representation with prime conductor $p$ such that $\chi=\det\rho$ is odd. Assume that $\rho$ is not dihedral.
   It was shown by Serre in \cite[Theorem 7]{Se} that
    \begin{itemize}
        \item[(a)] $p\not\equiv 1 \pmod{8}$;
        \item[(b)] if $p\equiv 5 \pmod{8}$, then $\rho$ is of type $S_4$ (i.e., $\bar{\rho}(G_\QQ) \simeq S_4$), and $\chi$ has order $4$ and conductor $p$;
        \item[(c)] if $p\equiv 3 \pmod{4}$, then $\rho$ is of type $S_4$ or $A_5$, and $\chi$ is the Legendre symbol $n\mapsto \left(\dfrac{n}{p}\right)$.
    \end{itemize}
    In addition, Serre also proves the following on \cite[page 250]{Se}: 
    The image $\rho(G_\QQ)$ consists of all elements $s\in \gl_2(\CC)$ whose image $\bar{s}\in \pgl_2(\CC)$ lies in $\bar{\rho}(G_\QQ)$ such that
    \begin{itemize}
         \item $\det(s)^2=\operatorname{sgn}(\bar{s})$ if $p\equiv 5 \pmod{8}$; 
        \item $\det(s)=\operatorname{sgn}(\bar{s})$ if $p\equiv 3 \pmod{4}$ and $\rho$ is of type $S_4$;
        \item $\det(s)=\pm 1$ if $p\equiv 3  \pmod{4}$ and $\rho$ is of type $A_5$.
    \end{itemize}
    Hence if the newform $f$ is of $S_4$-type and the level $N = p$ is a prime, then $f$ satisfies the assumption of Theorem \ref{thm:S4_refinement}(ii), and we conclude that
    \[
    K_f = \begin{cases}
        \QQ(\zeta_4) & \textrm{ if } \,\, p\equiv 5 \pmod{8}, 
        \\
        \QQ(\sqrt{-2}) & \textrm{ if } \,\, p\equiv 3\pmod{4}. 
    \end{cases}
    \]
\end{rem}

\subsubsection{Preliminaries for the proof of Theorem \ref{thm:S4_refinement}}

Before proving Theorem \ref{thm:S4_refinement}, we introduce a bit more notation and make a few observations.

\begin{dfn}
    For any finite order character  $\psi \colon G_\QQ \to \CC^\times$ of conductor dividing $N$ and $c \in \CC^\times$, we define the set $\cP_N(\psi = c)$ of primes by 
    \[
    \cP_N(\psi = c) := \{p \in \cP_N \mid \psi(p) = c\}. 
    \]
\end{dfn}

Since $f$ is of $S_4$-type, recall that  $\mathcal{P}_N = R_1\sqcup R_2 \sqcup R_3\sqcup R_4$, with Dirichlet densities 
\[
d(R_1)=\frac{1}{24}, \quad d(R_2)=\frac{9}{24},\quad \ d(R_3)=\frac{8}{24},\quad   d(R_4)=\frac{6}{24}.
\]
The set $R_2$ can be further decomposed as $R_2 = R_2^+ \sqcup R_2^-$, where 
\[
R_2^{\pm} := R_2 \cap \cP_N(\operatorname{sgn}\circ \bar{\rho}_f = \pm1).
\]
The corresponding Dirichlet densities are given by
\[
d(R_2^+)=\frac{3}{24}, \quad d(R_2^-)=\frac{6}{24}.
\]

The following two lemmas follow immediately from the definitions.

\begin{lem}\label{lem:density_sgn}
    We have 
    \begin{align*}
    \cP_N(\operatorname{sgn}\circ \bar{\rho}_f = 1) =  R_1 \sqcup R_2^+\sqcup R_3 \quad \textrm{ and } \quad 
    \cP_N(\operatorname{sgn}\circ \bar{\rho}_f = -1) =  R_2^- \sqcup R_4. 
    \end{align*}
\end{lem}

Let $k := \mathrm{ord}_2(d)$ denote the $2$-adic order of $d$.  

\begin{lem}\label{lem:density_chi}
    $Q_{2^k} = \cP_N(\chi^{d/2} = -1)$. 
\end{lem}



\subsubsection{Proof of Theorem \ref{thm:S4_refinement}(i)}\label{sec:proof_of_S_4_refinement_i}
We  assume that $\chi^{d/2} \neq \mathrm{sgn} \circ \bar{\rho}_f$. 
From the proof of Theorem \ref{thm:mainS4general} (see in particular (i--a) and (ii--a) in the proof), it suffices to show that
$Q_{2^k} \cap (R_1 \sqcup R_3) \neq \emptyset$.

Let $M$ be the fixed field of $\ker(\chi^{d/2}) \cap \ker(\mathrm{sgn} \circ \bar{\rho}_f)$. Since $\chi^{d/2} \neq \mathrm{sgn} \circ \bar{\rho}_f$ by assumption, it follows that $M/\QQ$ is a Galois extension and 
\[
\mathrm{Gal}(M/\QQ) \stackrel{\sim}{\to}
\ZZ/2\ZZ \times \ZZ/2\ZZ; \, \mathrm{Frob}_p \mapsto (\chi^{d/2}(p), (\mathrm{sgn} \circ \bar{\rho}_f)(p)). 
\]
Hence the Chebotarev density theorem, together with Lemmas \ref{lem:density_sgn} and \ref{lem:density_chi}, implies that 
\[
d(Q_{2^k} \cap (R_1 \sqcup R_2^+\sqcup R_3)) = \frac{1}{4}. 
\]
Since $d(R_2^+) = 3/24$, we deduce that $d(Q_{2^k} \cap (R_1 \sqcup R_3)) > 0$ and in particular, $Q_{2^k} \cap (R_1 \sqcup R_3) \neq \emptyset$.

\subsubsection{Proof of Theorem \ref{thm:S4_refinement}(ii)}\label{sec:proof_of_S_4_refinement_ii}

We assume that $\chi^{d/2} = \mathrm{sgn} \circ \bar{\rho}_f$. 
From the second paragraph of the proof of Theorem \ref{thm:mainS4general}, we obtain that $\zeta_{d'} \in  K_f$ with $d' := d/2^k$.

Since $\chi^{d/2}=\operatorname{sgn}\circ \bar{\rho}$ by assumption, Lemmas \ref{lem:density_sgn} and \ref{lem:density_chi} imply that 
\begin{align*}
    \mathcal{P}_N \setminus Q_{2^k} = R_1\sqcup R_2^+\sqcup R_3 \quad  \text{ and } \quad 
    Q_{2^k} = R_2^- \sqcup R_4.
\end{align*}
Hence, for any prime $p \in \mathcal{P}_N \setminus Q_{2^k}$, Lemma \ref{lem:hecke_field_ap} yields 
\[
\QQ(a_p(f)) \subset \QQ( \zeta_{d'}, \zeta_{2^k}) \subset \QQ( \zeta_{d'},  \sqrt{2}\zeta_{2^{k+1}}). 
\]
Moreover, since $R_4 \subset Q_{2^k}$, Lemma \ref{lem:hecke_field_ap} once again gives, for any prime $p \in R_4$, 
\[
 \QQ( \sqrt{2}\zeta_{2^{k+1}}) \subset
\QQ(a_p(f))  \subset \QQ( \zeta_{d'},  \sqrt{2}\zeta_{2^{k+1}}). 
\]
Finally, since $a_p(f) = 0$ for any prime $p \in R_2$ by Lemma \ref{lem:BL_c(g)}, combining these two facts with the decomposition $\mathcal{P}_N = (\mathcal{P}_N \setminus Q_{2^k}) \sqcup R_2^- \sqcup R_4$, we deduce that
\[
K_f = \QQ(\zeta_{d'}, \sqrt{2}\zeta_{2^{k+1}}) = 
\begin{cases}
\QQ(\zeta_{d'}, \sqrt{-2}) & \text{ if } \,  k=1, 
\\
\QQ(\zeta_{d'}, \zeta_4)  & \text{ if } \,  k=2.  
\end{cases}
\]

\section{On a generator of the Hecke field $K_f$}

Let $f=\sum_{n\ge 1} a_n(f)q^n \in S_1(N,\chi)$ be a weight one exotic newform of level $N$
and nebentypus $\chi$ with $\chi(-1)=-1$. 
By the proofs of Theorems \ref{thm:mainA4general}, \ref{thm:mainA5pf},
and \ref{thm:mainS4general}, we know that
\[
K_f = \bQ\bigl(\{\, a_p(f) \mid p \gg N \,\}\bigr).
\]
In this section, we consider the set of primes 
\[
\cP_f := \{p \mid \gcd(p,N) = 1 \, \textrm{ and } \, K_f = \bQ(a_p(f))\}, 
\]
and compute its Dirichlet density. 
In what follows, let $\varphi$ denote the Euler's totient function.

\begin{thm}\label{thm:density_A_4_case}
Suppose that $f$ is of $A_4$-type. 
Then the Dirichlet density $d(\cP_f)$ of the set $\cP_f$ is given by 
\begin{align*}
    d(\cP_f) = 
    \begin{cases}
        \dfrac{3 \varphi(d)}{4d} & \textrm{if} \quad    \ker(\chi) \not\subset \ker(G_\bQ \stackrel{\overline{\rho}_f}{\twoheadrightarrow} A_4 \twoheadrightarrow A_4^{\rm ab}), 
        \\
        \\
        \dfrac{\varphi(d)}{d} &   \textrm{if} \quad \ker(\chi) \subset \ker(G_\bQ \stackrel{\overline{\rho}_f}{\twoheadrightarrow} A_4 \twoheadrightarrow A_4^{\rm ab}). 
    \end{cases}
\end{align*}
\end{thm}

\begin{proof}
Let $L_1$ denote the field corresponding to the open subgroup
$\ker\bigl(G_{\bQ} \stackrel{\overline{\rho}_f}{\twoheadrightarrow} A_4\bigr)$,
and let $L_2$ denote the field corresponding to the open subgroup
$\ker\bigl(G_{\bQ} \stackrel{\chi}{\twoheadrightarrow} \mu_d\bigr)$,
where $\mu_d \subset \bC$ denotes the group of $d$-th roots of unity.
Then we have a surjective homomorphism 
\begin{align}\label{eq:galois_image_quoutient_by_center}
G_\bQ \twoheadrightarrow \Gal(L_1L_2/\bQ) \simeq \Gal(L_1/\bQ) \times_{\Gal((L_1 \cap L_2)/\bQ)}  \Gal(L_2/\bQ).      
\end{align}
Since $L_2/\bQ$ is abelian, the extension $L_1 \cap L_2$ over $\bQ$ is also abelian.
On the other hand,
\[
\Gal(L_1/\bQ)^{\mathrm{ab}} \simeq A_4^{\mathrm{ab}} \simeq \bZ/3\bZ.
\]
It follows that
\[
[L_1 \cap L_2 : \bQ] = 1 \, \text{ or } \, 3.
\]
Suppose  that $\ker(\chi) \not\subset \ker(G_\bQ \stackrel{\overline{\rho}_f}{\twoheadrightarrow} A_4 \twoheadrightarrow A_4^{\rm ab})$. 
In this case, we have $L_1 \cap L_2 = \bQ$. Hence it follows from \eqref{eq:galois_image_quoutient_by_center} that there exists a surjective homomorphism
\[
\widetilde{\rho}_f \colon G_{\bQ} \twoheadrightarrow A_4 \times \mu_d
\]
whose composition with the first (resp.\ second) projection agrees with
$\overline{\rho}_f$ (resp. $\chi$). 
Since $K_f = \bQ(\zeta_{2d})$ by Theorem \ref{thm:mainA4general}, it follows from Lemma \ref{lem:hecke_field_ap}  that 
$p \in \cP_f$ if and only if 
$\mathrm{pr}_1(\widetilde{\rho}_f(\mathrm{Frob}_p)) \in A_4$ has order $1$ or $3$ 
and $\mathrm{pr}_2(\widetilde{\rho}_f(\mathrm{Frob}_p)) \in \mu_d$ is a generator.  The proportion of such elements is $3 \varphi(d)/4d$, 
and therefore, by the Chebotarev density theorem, we have 
\[
d(\cP_f) = \frac{3 \varphi(d)}{4d}.
\]
Next, suppose that $\ker(\chi) \subset \ker(G_\bQ \stackrel{\overline{\rho}_f}{\twoheadrightarrow} A_4 \twoheadrightarrow A_4^{\rm ab})$. 
Then we have
\[
\Gal((L_1 \cap L_2)/\bQ) \simeq \bZ/3\bZ.
\]
It follows from \eqref{eq:central_extension_Z/a_G_Gbar} that there exists a surjective homomorphism
\[
\widetilde{\rho}_f \colon G_\bQ \twoheadrightarrow A_4 \times_{\bZ/3\bZ} \mu_d
\]
whose composition with the first (resp.\ second) projection agrees with
$\overline{\rho}_f$ (resp.\ $\chi$).
In this case, for any prime $p \nmid N$, the element $\mathrm{pr}_1\bigl(\widetilde{\rho}_f(\mathrm{Frob}_p)\bigr) \in A_4$ has order $3$ whenever $\mathrm{pr}_2\bigl(\widetilde{\rho}_f(\mathrm{Frob}_p)\bigr) \in \mu_d$ is a generator. 
Since $K_f = \bQ(\zeta_{2d})$ by Theorem \ref{thm:mainA4general}, it follows from Lemma \ref{lem:hecke_field_ap} that
$p \in \cP_f$ if and only if $\mathrm{pr}_2\bigl(\widetilde{\rho}_f(\mathrm{Frob}_p)\bigr) \in \mu_d$ is a generator. 
Therefore, the proportion of such elements is $\varphi(d)/d$, and by the Chebotarev density theorem, 
\[
d(\cP_f) = \frac{\varphi(d)}{d}.
\]
\end{proof}

\begin{thm}\label{thm:density_A_5_case}
Suppose that $f$ is of $A_5$-type. 
Then the Dirichlet density $d(\cP_f)$ of the set $\cP_f$ is given by 
\begin{align*}
    d(\cP_f) = 
    \begin{cases}
        \dfrac{2 \varphi(d)}{5d} & \textrm{if} \quad  5 \nmid d, 
        \\
        \\
        \dfrac{3\varphi(d)}{4d} &  \textrm{if} \quad  5 \mid d. 
    \end{cases}
\end{align*}
\end{thm}

\begin{proof}
Since the abelianization of $A_5$ is trivial, by the same argument as in the proof of Theorem~\ref{thm:density_A_4_case},
there exists a surjective homomorphism
\[
\widetilde{\rho}_f \colon G_\bQ \twoheadrightarrow A_5 \times \mu_d
\]
whose composition with the first (resp.\ second) projection agrees with
$\overline{\rho}_f$ (resp.\ $\chi$). 
Since $K_f = \bQ(\zeta_{2d}, \sqrt{5})$ by Theorem \ref{thm:mainA4general}, 
it follows from Lemma \ref{lem:hecke_field_ap} that $p \in \cP_f$ if and only if 
$\mathrm{pr}_2(\widetilde{\rho}_f(\mathrm{Frob}_p)) \in \mu_d$ is a generator and
\[
\mathrm{pr}_1(\widetilde{\rho}_f(\mathrm{Frob}_p)) \in A_5 \, \text{ has order } 
\begin{cases}
5, & \text{if } 5 \nmid d, \\
1, 3, \text{ or } 5, & \text{if } 5 \mid d.
\end{cases}
\]
Since the proportion of such elements is $2\varphi(d)/5d$ if $5 \nmid d$, and $3\varphi(d)/4d$ if $5 \mid d$, it  follows from the Chebotarev density theorem that $d(\cP_f)$ equals this proportion. 
\end{proof}

\begin{thm}\label{thm:density_S_4_case}
Suppose that $f$ is of $S_4$-type. 
\begin{itemize}
    \item[(i)] If $\mathrm{sgn} \circ \overline{\rho}_f  \neq \chi^{d/2}$ and $\mathrm{ord}_2(d) = 1$, 
then $\cP_f = \emptyset$, i.e., the Hecke field $K_f$ is not generated, as a $\bQ$-algebra, by a single $a_p(f)$. 
    \item[(ii)] If $\mathrm{sgn} \circ \overline{\rho}_f  \neq \chi^{d/2}$ and $\mathrm{ord}_2(d) \geq 2$, 
then  the Dirichlet density $d(\cP_f)$ of the set $\cP_f$  is given by 
\begin{align*}
    d(\cP_f) = 
    \begin{cases}
        \dfrac{3 \varphi(d)}{8d} & \textrm{if} \quad  \mathrm{ord}_2(d) = 2, 
        \\
        \\
        \dfrac{5\varphi(d)}{8d} &  \textrm{if} \quad  \mathrm{ord}_2(d) \geq 3. 
    \end{cases}
\end{align*}
    \item[(iii)] If $\mathrm{sgn} \circ \overline{\rho}_f  = \chi^{d/2}$, 
then  the Dirichlet density $d(\cP_f)$ of the set $\cP_f$  is given by 
\begin{align*}
    d(\cP_f) = 
    \begin{cases}
        \dfrac{\varphi(d)}{2d} & \textrm{if} \quad  \mathrm{ord}_2(d) \neq 2, 
        \\
        \\
        \dfrac{7\varphi(d)}{8d} &  \textrm{if} \quad  \mathrm{ord}_2(d) = 2. 
    \end{cases}
    \end{align*}
\end{itemize}
\end{thm}

\begin{proof}
If $\mathrm{sgn} \circ \overline{\rho}_f  \neq \chi^{d/2}$ and $\mathrm{ord}_2(d) = 1$, then by Theorem \ref{thm:S4_refinement} we have
$K_f = \bQ(\zeta_{4d})$.  
However, Lemma \ref{lem:hecke_field_ap} shows that
$\zeta_8 \notin \bQ(a_p(f))$ since $\mathrm{ord}_2(d) = 1$.  
In particular, $\bQ(a_p(f)) \neq K_f$ for any prime $p \nmid N$, which proves claim (i).

Let us prove claim (ii). If $\mathrm{sgn} \circ \overline{\rho}_f  \neq \chi^{d/2}$ and $\mathrm{ord}_2(d) \geq 2$, then 
the same argument as in the proof of Theorem~\ref{thm:density_A_4_case},
there exists a surjective homomorphism 
\[
\widetilde{\rho}_f \colon G_\bQ \twoheadrightarrow S_4 \times \mu_d
\]
whose composition with the first (resp.\ second) projection agrees with
$\overline{\rho}_f$ (resp. $\chi$).  
Take a prime $p \nmid N$.  
Since $K_f = \bQ(\zeta_{2d})$ by  Theorems \ref{thm:mainS4general} and \ref{thm:S4_refinement}, 
it follows from Lemma \ref{lem:hecke_field_ap} that $p \in \cP_f$ if and only if 
$\mathrm{pr}_2(\widetilde{\rho}_f(\mathrm{Frob}_p)) \in \mu_d$ is a generator and
\[
\mathrm{pr}_1(\widetilde{\rho}_f(\mathrm{Frob}_p)) \in S_4 \, \text{ has order } 
\begin{cases}
1 \text{ or } 3, & \text{if } \mathrm{ord}_2(d) = 2, \\
1, 3, \text{ or } 4, & \text{if } \mathrm{ord}_2(d) \geq 3.
\end{cases}
\]
Since the proportion of such elements is $3\varphi(d)/8d$ if $\mathrm{ord}_2(d) = 2$, and $5\varphi(d)/8d$ if $\mathrm{ord}_2(d) \geq 3$, it  follows from the Chebotarev density theorem that $d(\cP_f)$ equals this proportion.

Finally, assume that $\mathrm{sgn} \circ \overline{\rho}_f = \chi^{d/2}$. 
An argument identical to that of the proof of Theorem~\ref{thm:density_A_4_case}
yields a surjective homomorphism
\[
\widetilde{\rho}_f \colon G_\bQ \twoheadrightarrow S_4 \times_{\bZ/2\bZ} \mu_d
\]
whose composition with the first (resp.\ second) projection agrees with
$\overline{\rho}_f$ (resp.\ $\chi$). 
Take a prime $p \nmid N$.  
Note that if $\mathrm{pr}_2(\widetilde{\rho}_f(\mathrm{Frob}_p)) \in \mu_d$ is a generator, then $\mathrm{sgn}(\overline{\rho}_f(\mathrm{Frob}_p)) = -1$. 
This shows that $\overline{\rho}_f(\mathrm{Frob}_p) \in S_4$ has order $2$ or $4$.  
Hence when $\mathrm{ord}_2(d) \neq 2$, Theorems \ref{thm:mainS4general} and \ref{thm:S4_refinement}, together with
Lemma \ref{lem:hecke_field_ap}, imply that
$p \in \cP_f$ if and only if
$\mathrm{pr}_2(\widetilde{\rho}_f(\mathrm{Frob}_p)) \in \bZ/d\bZ$ is a generator and
$\mathrm{pr}_1(\widetilde{\rho}_f(\mathrm{Frob}_p)) \in S_4$ has order $4$.  
Since the proportion of such elements is $\varphi(d)/2d$,
it follows from the Chebotarev density theorem that
\[
d(\cP_f) = \frac{\varphi(d)}{2d}.
\]
When $\mathrm{ord}_2(d) = 2$, Theorems \ref{thm:mainS4general} and \ref{thm:S4_refinement} show that $K_f = \bQ(\zeta_d)$. 
In this case, $p \in \cP_f$ if and only if 
\begin{itemize}
    \item $\mathrm{pr}_2(\widetilde{\rho}_f(\mathrm{Frob}_p)) \in \bZ/d\bZ$ is a generator and
$\mathrm{pr}_1(\widetilde{\rho}_f(\mathrm{Frob}_p)) \in S_4$ has order $4$, or 
\item $\mathrm{pr}_2(\widetilde{\rho}_f(\mathrm{Frob}_p)) \in \bZ/d\bZ$ has order $d/2$ and
$\mathrm{pr}_1(\widetilde{\rho}_f(\mathrm{Frob}_p)) \in S_4$ has order $1$ or $3$.  
\end{itemize}
Since the proportion of such elements is $7\varphi(d)/8d$,
it follows from the Chebotarev density theorem that $d(\cP_f) = 7\varphi(d)/8d$. 
\end{proof}

\begin{rem}\label{rem:density_Q(chi)(a_p(f))}
Since the newform $f$ has nebentypus $\chi$, one should perhaps consider the field
$\bQ(\chi)(a_p(f))$ instead of $\bQ(a_p(f))$, where $\bQ(\chi) := \bQ(\mathrm{Im}(\chi)) = \bQ(\zeta_d)$. 
In other words, we consider the Dirichlet density of the set of primes
\[
\cP_f^{\chi}
:= \{ p \mid \gcd(p,N)=1 \, \text{ and } \, K_f = \bQ(\chi)(a_p(f)) \}.
\]
By the same argument as in Theorems \ref{thm:density_A_4_case}, \ref{thm:density_A_5_case}, and \ref{thm:density_S_4_case}, the following result holds.
\begin{itemize}
    \item[(i)] If $f$ is $A_4$-type, then  
    \begin{align*}
    d(\cP_f^{\chi}) =  \dfrac{3}{8}
    \end{align*}
        \item[(ii)] If $f$ is $A_5$-type, then  
        \[
            d(\cP_f^{\chi}) =     \begin{cases}
        \dfrac{1}{5} & \textrm{if} \quad  5 \nmid d, 
        \\
        \\
        \dfrac{3}{8} &  \textrm{if} \quad  5 \mid d. 
    \end{cases}
        \]
      \item[(iii)] If $f$ is $S_4$-type, then  
        \[
            d(\cP_f^{\chi}) =     \begin{cases}
        0 & \textrm{if} \quad \mathrm{sgn} \circ \overline{\rho}_f  \neq \chi^{d/2} \textrm{ and } \mathrm{ord}_2(d) = 1, 
        \\
        \\
        \dfrac{3}{16} &  \textrm{if} \quad \mathrm{sgn} \circ \overline{\rho}_f  \neq \chi^{d/2} \textrm{ and } \mathrm{ord}_2(d) = 2, 
                \\
        \\
        \dfrac{5}{16} &  \textrm{if} \quad \mathrm{sgn} \circ \overline{\rho}_f  \neq \chi^{d/2} \textrm{ and } \mathrm{ord}_2(d) \geq 3, 
                \\
        \\
         \dfrac{1}{4} &  \textrm{if} \quad \mathrm{sgn} \circ \overline{\rho}_f  = \chi^{d/2} \textrm{ and } \mathrm{ord}_2(d) \neq 2,  
                \\
        \\
        1 &  \textrm{if} \quad \mathrm{sgn} \circ \overline{\rho}_f  = \chi^{d/2} \textrm{ and } \mathrm{ord}_2(d) = 2.  
    \end{cases}
        \]
\end{itemize}
\end{rem}

\section{Strongly minimal newforms}\label{sec:Strongly minimal newforms}

For any Dirichlet character $\psi$, there exists a unique newform
$g := f \otimes \psi$, called the twist of $f$ by $\psi$, characterized by the relation $a_p(g) = \psi(p) a_p(f)$ 
for almost all primes $p$. 
 In this case we say that $f$ and $g$ are twist equivalent.

\begin{dfn}
A newform is said to be twist-minimal if its level attains the minimal value in its twist class.
\end{dfn}

As noted in Remark \ref{rem_twists}, twisting by Dirichlet characters allows one to vary the order of the nebentypus of a newform almost arbitrarily. It is therefore natural to restrict our attention to twist-minimal newforms. 
However, the orders of the nebentypus of twist-minimal newforms that are twist-equivalent need not coincide.
Accordingly, among twist-minimal newforms in a given twist class, we shall focus on one whose nebentypus has minimal order.

\begin{dfn}\label{def:strongly_minimal}
We say that a newform $f$ with nebentypus $\chi$ is strongly minimal if it satisfies the following conditions:
\begin{itemize}
\item[(i)] $f$ is twist-minimal;
\item[(ii)] for every twist-minimal newform $g$ that is twist-equivalent to $f$, with nebentypus $\chi_g$,  one has
$|\chi(I_p)| \leq |\chi_g(I_p)|$ for every prime $p$.
\end{itemize}
\end{dfn}

Note that each twist class contains a strongly minimal newform. 
By Lemma \ref{lem:order_character}, within any twist class, the nebentypus of a strongly minimal newform has the minimal order among the twist-minimal newforms.

Recall that $d$ denotes the order of  the nebentypus $\chi$. 
The following theorems are the main results of this section.

\begin{thm}\label{thm:minimal_order}
    Suppose that the weight one newform $f$ is strongly minimal.  
    \begin{itemize}
        \item[(i)] If $f$ is of $A_4$-type, then $d$ takes the form $3 \cdot 2^k$ with $k \geq 1$. 
        \item[(ii)] If $f$ is of $A_5$-type, then  $d$  takes the form $2^{k} \cdot 3^{\delta_3} \cdot 5^{\delta_5}$ with $k \geq 1$ and $\delta_3, \delta_5 \in \{0,1\}$. 
        \item[(iii)] If $f$ is of $S_4$-type, then $d$ takes the form $2^{k}$ or $3 \cdot 2^k$ with $k \geq 1$. 
    \end{itemize}
\end{thm}

The proof of this theorem is given in \S\ref{sec:proof_thm_minimal}.

 \subsection{Local lifting}

In this subsection, we fix a prime $p$ and let
\[
\bar{\rho}_p \colon G_{\QQ_p} \to \pgl_2(\CC)
\]
denote a projective Galois representation.

\begin{dfn}
We say that a lifting $\rho \colon G_{\QQ_p} \to \gl_2(\CC)$ of $\bar{\rho}_p$  is minimal if the conductor of $\rho$ is minimal within the set of liftings of $\bar{\rho}$ to $\gl_2(\CC)$. 
\end{dfn}

If we have  two lifts of $\bar{\rho}_p$ to $\gl_2(\CC)$, then they differ only by a twist by a character on $G_{\QQ_p}$. This fact will be used frequently below without further mention.

We now consider an explicit minimal lift  of $\bar{\rho}_p$ to $\mathrm{GL}_2(\mathbb{C})$.
First, recall the following well-known fact.

\begin{prop}\label{prop:conj}
Two finite subgroups of $\pgl_2(\CC)$ which are isomorphic are conjugate. 
\end{prop}
\begin{proof}
    See, for example, \cite[Proposition 4.1]{Beauville10}. 
\end{proof}

\begin{lem}\label{lem:lift_cyclic}
If $\bar{\rho}_p(G_{\QQ_p})$ is cyclic, then there is a minimal lifting $\rho_p \colon G_{\QQ_p} \to \gl_2(\CC)$ of $\bar{\rho}_p$ such that $\rho_p \simeq \det \rho_p \oplus \mathbbm{1}$. 
\end{lem}
\begin{proof}
 Let $\bar{\rho}_p(\sigma) \in \bar{\rho}_p(G_{\QQ_p})$ be a generator. 
 By Proposition \ref{prop:conj}, the element $\bar{\rho}_p(\sigma)$ is conjugate to $\begin{pmatrix}\zeta & 0 \\ 0 & 1\end{pmatrix} \bmod{ \mathbb{C}^\times}$, where $\zeta$ is a root of unity. 
Hence, one can choose a lift $g \in \mathrm{GL}_2(\mathbb{C})$ of $\bar{\rho}_p(\sigma)$ which is conjugate to $\begin{pmatrix}\zeta & 0\\ 0 & 1\end{pmatrix}$.
Then  $\mathrm{ord}(g) = \mathrm{ord}(\bar{\rho}_p(\sigma))$, and 
 the homomorphism 
\[
\rho_p \colon G_{\QQ_p} \to G_{\QQ_p}/\ker(\bar{\rho}_p) \stackrel{\sim}{\to} \langle g \rangle \subset \gl_2(\CC); \sigma \mapsto g
\]
is a desired lift of $\bar{\rho}_p$. The minimality of $\rho_p$ follows from the fact that $\rho_p \simeq \det \rho_p \oplus \mathbbm{1}$. 
\end{proof}

\begin{lem}\label{lem:lift_dihedral_odd}
If $\bar{\rho}_p(G_{\QQ_p}) \simeq D_{2n}$ for some odd positive integer $n$,  
    then there is a minimal lifting $\rho_p \colon G_{\QQ_p} \to \gl_2(\CC)$ of $\bar{\rho}_p$ such that $\det \rho_p$ is a  quadratic character. 
\end{lem}

\begin{proof}
Since $\bar{\rho}_p(G_{\mathbb{Q}_p}) \simeq D_{2n}$, we can define a Galois representation $\tilde{\rho}_p$ by 
\[
\tilde{\rho}_p \colon G_{\mathbb{Q}_p} \stackrel{\bar{\rho}_p}{\to} \bar{\rho}_p(G_{\mathbb{Q}_p}) \simeq
\left\{
\begin{pmatrix} \zeta_n^a & 0 \\ 0 & \zeta_n^{-a} \end{pmatrix},
\begin{pmatrix} 0 & \zeta_n^{-a} \\ \zeta_n^a & 0 \end{pmatrix}
\,\middle|\, 0 \le a < n
\right\} \subset \mathrm{GL}_2(\mathbb{C}).
\]
Since $n$ is odd, the dihedral group $D_{2n}$ has trivial center. It follows that the projective image of $\tilde{\rho}_p$ is isomorphic to $D_{2n}$.
By Proposition \ref{prop:conj}, this projective image is conjugate to $\bar{\rho}_p(G_{\mathbb{Q}_p})$, so, if needed, we may conjugate $\tilde{\rho}_p$ to obtain the desired lift $\rho_p$. 
By construction, the lift $\rho_p$ is minimal. 
\end{proof}

\begin{lem}\label{lem:lift_dihedral_even}
Suppose that $p$ is an odd prime. 
\begin{itemize}
    \item[(i)]  If $\bar{\rho}_p(G_{\QQ_p}) \simeq D_{8}$,  
     then there is a minimal lifting $\rho_p \colon G_{\QQ_p} \to \gl_2(\CC)$  such that $\det \rho_p$ is a (ramified) quadratic character.  
    \item[(ii)] Let $k := \mathrm{ord}_2(p-1)$. If $\bar{\rho}_p(G_{\QQ_p}) \simeq D_{4}$,  
    then there is a minimal lifting  $\rho_p \colon G_{\QQ_p} \to \gl_2(\CC)$ such that  $\det \rho_p$ is a totally ramified character of order $2^k$. 
    Furthermore, for any lifting $\rho_p'$ of $\bar{\rho}_p$, the order of $\det \rho_p'$ is divisible by $2^k$. 
    \end{itemize}
\end{lem}
\begin{proof}
Let us show claim (i).
Since $p \geq 3$, the image $\bar{\rho}_p(I_p)$ is cyclic. Moreover, the quotient $\bar{\rho}_p(G_{\QQ_p})/\bar{\rho}_p(I_p)$ is also cyclic. 
It then follows from $\bar{\rho}_p(G_{\QQ_p}) \simeq D_{8}$ that $|\bar{\rho}_p(I_p)| = 4$. 
Let $K$ be the unramified quadratic extension of $\mathbb{Q}_p$.
Then there exists a totally ramified character $\mu \colon G_K \to \mathbb{C}^{\times}$ of order $4$ such that
\[
\bar{\rho}_p \simeq \mathrm{Ind}_{K/\mathbb{Q}_p}(\mu).
\]
Since $p^2 \equiv 1 \pmod{8}$, we can choose a totally ramified character $\tilde{\mu} \colon G_K \to \mathbb{C}^{\times}$ with $\tilde{\mu}^2 = \mu$ and define
\[
\tilde{\rho}_p := \mathrm{Ind}_{K/\mathbb{Q}_p}(\tilde{\mu}) \colon G_{\mathbb{Q}_p} \to \mathrm{GL}_2(\mathbb{C}).
\]
Because $\mu^{\sigma} = \mu^{-1}$, where $\sigma$ is the generator of $\mathrm{Gal}(K/\mathbb{Q}_p)$, we have $\tilde{\mu}^{\sigma} = \tilde{\mu}^b$ for $b \in \{-1,3\}$ and the image  $\tilde{\rho}_p(G_{\bQ_p})$ is (conjugate to) the group 
\[
\left\{
\begin{pmatrix} \zeta_8^{a} & 0 \\ 0 & \zeta_8^{ab} \end{pmatrix},
\begin{pmatrix} 0 & \zeta_8^{ab} \\ \zeta_8^{a} & 0 \end{pmatrix}
\,\middle|\, 0 \leq a < 8
\right\} \subset \mathrm{GL}_2(\mathbb{C}).
\]
In particular, the projective image of $\tilde{\rho}_p$ is isomorphic to $D_8$.
Hence, by Proposition \ref{prop:conj}, we may conjugate $\tilde{\rho}_p$ if necessary to obtain the desired lift $\rho_p$. 
Since $\tilde{\mu}$ is a tamely ramified character that cannot be extended to $G_{\mathbb{Q}_{p}}$, the lifting $\rho_p$ is minimal. 

Next, we establish claim (ii).
We continue to let $K/\mathbb{Q}_p$ be the unramified quadratic extension and $\mathrm{Gal}(K/\mathbb{Q}_p) = \langle \sigma \rangle$. 
Since $p^2 \equiv 1 \pmod{2^{k+1}}$, one can take a totally ramified character $\mu \colon G_K \to \CC^{\times}$ of order $2^{k+1}$. 
We then define
\[
\tilde{\rho}_p := \mathrm{Ind}_{K/\mathbb{Q}_p}(\mu) \colon G_{\mathbb{Q}_p} \to \mathrm{GL}_2(\mathbb{C}).
\]
Since $p \equiv 1 \pmod{2^{k}}$ by the definition of the integer $k$, we see that $\mu^2$ is the restriction of a character of $G_{\QQ_{p}}$, and hence $(\mu^{\sigma})^2 = \mu^2$. Therefore, we have 
\[
\mu^{\sigma} = - \mu
\]
and the image  $\tilde{\rho}_p(G_{\bQ_p})$ is (conjugate to) the set 
\[
\left\{
\begin{pmatrix} \zeta_{2^{k+1}}^{a} & 0 \\ 0 & -\zeta_{2^{k+1}}^{a} \end{pmatrix},
\begin{pmatrix} 0 & \zeta_{2^{k+1}}^{a} \\ -\zeta_{2^{k+1}}^{a} & 0 \end{pmatrix}
\,\middle|\, 0 \leq b < 2^{k+1}
\right\} \subset \mathrm{GL}_2(\mathbb{C}).
\]
In particular, the projective image of $\tilde{\rho}_p$ is isomorphic to $D_4$ and by Proposition \ref{prop:conj}, we may conjugate $\tilde{\rho}_p$ if necessary to obtain the desired lift. 
For the same reason as in claim (i), the lifting $\rho_p$ is minimal. 
Since any lifting $\rho_p'$ of $\bar{\rho}_p$ can be written as $\rho_p \otimes \psi$, 
 the remaining assertion follows from the facts that  $\mathrm{ord}_2(p-1) = k$ and that  $\det \rho_p$ is a totally ramified character on $G_{\bQ_p}$ of order $2^k$. 
\end{proof}

\subsubsection{Global lifting}

First, we recall the result  of Tate on a lifting of $\bar{\rho}$ to $\gl_2(\CC)$, as presented in \cite{Se}.

\begin{thm}[{\cite[Theorem 5]{Se}}]\label{thm:tate_lift}
For each prime $p$, let  $\rho'_p \colon G_{\QQ_p} \to \gl_2(\CC)$ be a lifting  of $\bar{\rho}|_{G_{\QQ_p}}$. 
Assume that $\rho'_p$ is unramified for all but finitely many $p$. 
Then there exists a unique lifting $\rho \colon G_{\QQ} \to \gl_2(\CC)$ of $\bar{\rho}$ such that 
$\rho|_{I_p} = \rho'_p|_{I_p}$ for all prime $p$. 
\end{thm}

\begin{lem}\label{lem:cyclic-dihedral}
For any odd prime $p$, the group $\bar{\rho}(G_{\QQ_p})$ is either cyclic or dihedral. 
\end{lem}
\begin{proof}
If $\bar{\rho}$ is at most tamely ramified at $p$, then $\bar{\rho}(G_{\mathbb{Q}_p})$ is metacyclic, and the classification of finite subgroups of $\mathrm{PGL}_2(\mathbb{C})$ implies that it is  cyclic or dihedral.
Next, suppose that $\bar{\rho}$ is wildly ramified at $p$.
In this case, the non-trivial $p$-Sylow subgroup of $\bar{\rho}(G_{\mathbb{Q}_p})$ is normal; hence, again by the classification of finite subgroups of $\mathrm{PGL}_2(\mathbb{C})$, either $p=2$, or $\bar{\rho}(G_{\mathbb{Q}_p})$ is cyclic or dihedral. 
\end{proof}

Recall that $\rho_f \colon G_\QQ \to \gl_2(\CC)$ denotes the Galois representation associated with the weight one newform $f$.

\begin{prop}\label{prop:minimal_order}
Let $S$ (resp. $T$) denote the set of odd primes $p$ for which $\bar{\rho}_f(G_{\mathbb{Q}_p})$ is cyclic (resp. isomorphic to $D_4$). 
Put $t := \max_{p \in T}\mathrm{ord}_2(p-1)$. 
If $f$ is strongly minimal, then we have 
\[
\mathrm{ord}(\chi) = \lcm\{2^t, \, |\chi(I_2)|, \, |\bar{\rho}_f(I_p)| \ (p \in S)\}. 
\]
\end{prop}
\begin{proof}
Note that if $p \not \in S \cup T \cup \{2\}$, then $\bar{\rho}_f(G_{\mathbb{Q}_p})$ is non-abelian dihedral by Lemma \ref{lem:cyclic-dihedral}. 
Moreover, if $\bar{\rho}_f(G_{\mathbb{Q}_p})$ is non-abelian, then $\bar{\rho}_f$ is ramified at $p$, and hence only finitely many primes lie outside $S \cup T$. 

For each odd prime $p \not\in S \cup T$, we denote by $\rho_p \colon G_{\QQ_p} \to \gl_2(\CC)$ the minimal lifting of $\bar{\rho}_f|_{G_{\QQ_p}}$ constructed in Lemmas \ref{lem:lift_dihedral_odd} and \ref{lem:lift_dihedral_even}. 
Then there exists a  character $\psi_p$ on $G_{\bQ_p}$ such that $\rho_p \simeq \rho_f|_{G_{\QQ_p}} \otimes \psi_p$. 
Let $\psi$ be the Dirichlet character satisfying $\psi|_{I_p} = \psi_p|_{I_p}$ for any prime $p \not\in S \cup T \cup \{2\}$ and $\psi|_{I_p} = \mathbbm{1}$ for any prime $p \in S \cup T \cup \{2\}$. 
Then by the definition of $\rho_p$, the lifting $(\rho_f \otimes \psi)|_{G_{\QQ_p}}$ is minimal for any odd prime $p$. 
Moreover, since $f$ is strongly minimal, Lemma \ref{lem:order_character} implies that 
\[
\mathrm{ord}(\chi) = \mathrm{ord}(\chi\psi^2).
\]
When $p \not\in S \cup T \cup \{2\}$, the construction of $\rho_p$ shows that $|(\chi\psi^{2})(I_p)|$ is $1$ or $2$. 
When $p \in S$, the strong minimality of $\rho_f$ together with Lemma \ref{lem:lift_cyclic} implies that 
$(\rho_f \otimes \psi)|_{I_p} \simeq \chi\psi^2|_{I_p} \oplus \mathbbm{1}$. 
In particular, $|\chi(I_p)| = |\chi\psi^2(I_p)| = |\bar{\rho}_f(I_{p})|$. 
Thus it follows from Lemmas \ref{lem:order_character} and \ref{lem:lift_dihedral_even} that $\mathrm{ord}(\chi) = \lcm\{2^t, \, |\chi(I_2)|, \, |\bar{\rho}_f(I_p)| \ (p \in S)\}$.  
\end{proof}


\subsection{Proof of Theorem \ref{thm:minimal_order}}\label{sec:proof_thm_minimal}

Note that $\chi(-1) = -1$ implies that the order of $\chi$ is even.

\subsubsection{Proof of Theorem \ref{thm:minimal_order}(i): $A_4$-type case}
Suppose that $f$ is strongly minimal and of $A_4$-type. 
 By Proposition \ref{prop:minimal_order}, to prove Theorem \ref{thm:minimal_order}(i), it suffices to show that there exists a prime $p$ such that $\bar{\rho}_f(G_{\QQ_p})$ is cyclic and $|\bar{\rho}_f(I_p)| = 3$.  
 
 Since the image projective $\bar{\rho}_f(G_\bQ)$ is isomorphic to $A_4$, the fixed field of $\ker( \bar\rho_f)$ contains a cyclic cubic extension $K/\QQ$ corresponding to the Klein four-group in $A_4$. 
 Let $q$ be a prime ramified in $K$. 
 Then $\bar{\rho}_f(I_q)$ surjects onto $\mathrm{Gal}(K/\QQ)$. 
Hence $q$ is an odd prime and $\bar{\rho}_f(G_{\QQ_p})$ is cyclic or dihedral by Lemma \ref{lem:cyclic-dihedral}. 
However, since the only metacyclic subgroups of $A_4$ that admit a surjection onto $\ZZ/3\ZZ$ are cyclic of order $3$, the group $\bar{\rho}_f(G_{\QQ_p})$ is cyclic of order $3$.

\subsubsection{Proof of Theorem \ref{thm:minimal_order}(ii): $A_5$-type case}  
Since any cyclic subgroup of $A_5$ has order $1$, $2$, $3$, or $5$, Theorem  \ref{thm:minimal_order}(ii) follows from Proposition \ref{prop:minimal_order}.

\subsubsection{Proof of Theorem \ref{thm:minimal_order}(iii): $S_4$-type case} 

Since any cyclic subgroup of $S_4$ has order $1$, $2$, $3$, or $4$, Theorem \ref{thm:minimal_order}(iii) follows from Proposition \ref{prop:minimal_order}.

\subsection{Strongly minimal newforms with a prescribed order of the nebentypus}

In Theorem \ref{thm:minimal_order}, we established the necessary conditions for the order of the nebentypus of a strongly minimal newform. 
We now turn to the question of whether, for each type, there exists a strongly minimal newform whose nebentypus has the specified order. 
Here we address this problem in the cases of the $A_4$-type and the $S_4$-type.

\subsubsection{$A_4$-type case}

Let $K := \QQ(\zeta_9 + \zeta_9^{-1})$ denote the maximal totally real subfield of $\QQ(\zeta_9)$. 
Note that $K/\QQ$ is an abelian extension with Galois group 
$\mathrm{Gal}(K/\QQ) \simeq \ZZ/3\ZZ$ and that the class number of $K$ is $1$. Hence, every ideal in $K$ is principal. 

For any local or global field $L$, let $\cO_L$ denote the ring of integers in $L$. 

\begin{lem}\label{lem:pm_square_existence}
    For any element $r \in \cO_K \setminus 2\cO_K$, 
    at least one of the congruences $X^2 \equiv \pm r \pmod{4\cO_K}$  admits a solution in $\cO_K$. 
\end{lem}
\begin{proof}
The prime $2$ does not split in $K$.
Thus, $K_2 := K \otimes_{\bQ} \QQ_2$ is the unramified extension of $\QQ_2$ of degree $3$.
It follows that
\[
(\cO_{K_2}/4\cO_{K_2})^\times \otimes_{\bZ} \bF_2 \simeq (\mathbb{F}_8^\times \times (1+2\cO_{K_2})/(1+4\cO_{K_2})) \otimes_{\bZ} \bF_2 \simeq (1+2\cO_{K_2})/(1+4\cO_{K_2}).
\]
Since the group $(1+2\cO{K_2})/(1+4\cO_{K_2})$ is generated by  $-1$, we deduce that at least one of the congruences $X^2 \equiv \pm r \pmod{4\cO_K}$ has a solution. 
\end{proof}

Let $\sigma$ denote a generator of the Galois group $\gal(K/\QQ)$. 
By Dirichlet's unit theorem (see \cite[Proposition 8.7.2]{NSW08}), there is  an isomorphism of $\QQ[\gal(K/\QQ)]$-modules 
\[
(\cO_K^\times \otimes_{\bZ} \QQ) \oplus \QQ \simeq \QQ[\gal(K/\QQ)]. 
\]
Hence, one can choose units $\epsilon \in \cO_K^\times$  such that 
\[
 \epsilon\epsilon^\sigma\epsilon^{\sigma^2} = 1, \quad  \cO_K^\times = \{\pm  \epsilon^{a_1} (\epsilon^{\sigma})^{a_2} \mid a_1, a_2 \in \ZZ \}. 
\]
Note that $K \simeq \QQ[x]/(x^3-3x-1)$. 
The cubic equation $x^3 - 3x - 1 = 0$ has two negative roots and one positive root, and these roots are units of norm $1$.
Therefore, the unit $\epsilon$ is neither totally positive nor totally negative.

\begin{lem}\label{lem:A4-extension}
Let $k$ be a positive ingteger. 
There exist infinitely many  primes $q>3$ with  $\mathrm{ord}_2(q-1) = k$ for which an odd projective Galois representation $\bar{\rho} \colon G_\QQ \to \pgl_2(\CC)$ (i.e., $\overline{\rho}(c) \neq 1$) exists satisfying the following properties:
    \begin{itemize}
    \item[(a)] $\bar{\rho}(G_\QQ) \simeq A_4$, 
    \item[(b)] $\bar{\rho}$ is unramified outside $3$, $q$, and $\infty$,  
   \item[(c)] $\bar{\rho}(G_{\QQ_3}) = \bar{\rho}(I_{\QQ_3}) \simeq \ZZ/3\ZZ$,  
       \item[(d)] $\bar{\rho}(G_{\QQ_q}) \simeq D_4$.  
    \end{itemize}
\end{lem}
\begin{proof}
Consider the field 
\[
M := K(\sqrt{\epsilon}, \sqrt{\epsilon^\sigma}, \zeta_{2^{k+1}}). 
\]
Since $\gal(K(\sqrt{\epsilon}, \sqrt{\epsilon^\sigma})/\QQ) \simeq A_4$ by the choice of the unit $\epsilon$, 
we have $K(\sqrt{\epsilon}, \sqrt{\epsilon^\sigma}) \cap K(\zeta_{2^{k+1}}) = K$, so that
\begin{align*}
    \gal(M/K) &\simeq \gal(K(\sqrt{\epsilon})/K) \times \gal(K(\sqrt{\epsilon^\sigma})/K) \times \gal(\QQ(\zeta_{2^{k+1}})/\QQ)
    \\
    &\simeq \ZZ/2\ZZ \times \ZZ/2\ZZ \times (\ZZ/2^{k+1}\ZZ)^\times. 
\end{align*}
Therefore, by the Chebotarev density theorem, there are infinitely many  primes $q>3$ such that the Frobenius conjugacy class at $q$ is equal to the set 
\[
\{(0,1,1+2^{k}), (1,0,1+2^{k}), (1,1,1+2^{k})\} \subset \gal(M/K) \subset \gal(M/\QQ)
\]
under the above isomorphism. 
By construction, $q \equiv 1 + 2^k \pmod{2^{k+1}}$, which implies that  $\mathrm{ord}_2(q-1) = k$.

The prime $q$ splits completely in $K$, so there exists a prime element $Q \in \cO_K$ satisfying $\pm q = Q Q^\sigma Q^{\sigma^2}$. 
Due to the choice of  $q$, the Frobenius automorphism $\mathrm{Frob}_Q \in \gal(K(\sqrt{\epsilon}, \sqrt{\epsilon^\sigma})/K)$ is non-trivial.
It follows that among the following congruences, exactly two have solutions (in $\cO_K$), while the other two do not: 
\begin{align*}
    &X^2 \equiv 1 \pmod{Q\cO_K},  &&X^2 \equiv \epsilon \pmod{Q\cO_K}, 
    \\
&X^2 \equiv \epsilon^{\sigma} \pmod{Q\cO_K},  &&X^2 \equiv \epsilon^{\sigma^2} \pmod{Q\cO_K}.   
\end{align*}
Since $\epsilon$ is  neither totally positive nor totally negative and 
\[
(\epsilon^{\sigma^i} Q)^{\sigma}(\epsilon^{\sigma^i} Q)^{\sigma^2} = (\epsilon^{\sigma^{i}})^{-1} Q^\sigma Q^{\sigma^2}, 
\]
There exists a unit $\epsilon' \in \{1, \epsilon, \epsilon^{\sigma}, \epsilon^{\sigma^2}\}$ such that $(\epsilon' Q)^\sigma (\epsilon' Q)^{\sigma^2}$ is neither totally positive nor totally negative and the congruence
\[
X^2 \equiv (\epsilon' Q)^\sigma (\epsilon' Q)^{\sigma^2} \pmod{Q\cO_K}
\]
has no solution. 
By replacing $Q$ with $\epsilon' Q$, we may assume that $\epsilon' = 1$.

Moreover, if necessary, replacing $Q$ with $-Q$ and applying Lemma \ref{lem:pm_square_existence} allows us to assume that the congruence
$X^2 \equiv Q \pmod{4\cO_K}$ has a solution. Note that $Q^\sigma Q^{\sigma^2}$ remains unchanged.

    Under the above preparations, set $L := K(\sqrt{QQ^{\sigma}}, \sqrt{Q^\sigma Q^{\sigma^2}})$. 
    Since
    \[
    \sqrt{QQ^{\sigma}}\sqrt{Q^{\sigma}Q^{\sigma^2}}\sqrt{Q^{\sigma^2}Q} = q, 
    \]
    the extension $L/\QQ$ is Galois, and we have an isomorphism $\mathrm{Gal}(L/\QQ) \simeq A_4$. 
Hence we obtain the projective Galois representation 
\[
\bar{\rho} \colon G_\QQ \to \gal(L/\QQ) \simeq A_4 \hookrightarrow \pgl_2(\CC). 
\]
Since $Q^\sigma Q^{\sigma^2}$ is neither totally positive nor totally negative, 
the field $L$ is totally imaginary, and hence $\bar{\rho}$ is odd.

Since the congruence $X^2 \equiv Q \pmod{4\cO_K}$ has a solution, the prime $2\cO_K$ is unramified in $L$. 
Hence condition (b) follows from the construction of the field $L$. 
Among the metacyclic subgroups of $A_4$, those whose order is a multiple of $3$ are isomorphic to $\mathbb{Z}/3\mathbb{Z}$.
From this observation, condition (c) follows since $3$ is ramified in $K$. 
Since the prime $q$ splits completely in $K$, 
the subgroup $\bar{\rho}(G_{\QQ_q})$ is contained in $\gal(L/K)$. 
The extension $K(\sqrt{Q Q^{\sigma}})/K$ is ramified at $Q$. The prime element $Q$ does not split in $K(\sqrt{Q^{\sigma} Q^{\sigma^2}})/K$ since the congruence $X^2 \equiv Q^\sigma Q^{\sigma^2} \pmod{Q\cO_K}$ has no solution. 
These facts imply that $\bar{\rho}(G_{\QQ_q}) = \gal(L/K) \simeq D_4$. 
\end{proof}

\begin{prop}\label{prop:exicentece_A4_minimal}
    For any positive integer $k$, there exists a strongly minimal newform of $A_4$-type with nebentypus of order $3 \times 2^k$. 
\end{prop}
\begin{proof}
By Lemma \ref{lem:A4-extension}, there exists a prime $q>3$ with $\mathrm{ord}_2(q-1) = k$ and  an odd projective Galois representat $\bar{\rho} \colon G_\QQ \to \pgl_2(\CC)$  satisfying the following properties:
    \begin{itemize}
    \item[(a)] $\bar{\rho}(G_\QQ) \simeq A_4$, 
    \item[(b)] $\bar{\rho}$ is unramified outside $3$, $q$, and $\infty$. 
   \item[(c)] $\bar{\rho}(G_{\QQ_3}) = \bar{\rho}(I_{\QQ_3}) \simeq \ZZ/3\ZZ$. 
       \item[(d)] $\bar{\rho}(G_{\QQ_q}) \simeq D_4$, 
    \end{itemize}
Since $\bar{\rho}(G_\QQ) \simeq A_4$, it follows from a theorem of Langlands \cite{Langlands80} that  one can choose a strongly minimal newform $f$ satisfying $\bar{\rho}_f \simeq \bar{\rho}$. 
From Lemma \ref{lem:order_character} and condition (b), we obtain 
\[
\mathrm{ord}(\chi) = \lcm\{|\chi(I_3)|, |\chi(I_q)|\}. 
\]
 Lemma \ref{lem:lift_cyclic}, combined with condition (c), implies $|\chi(I_3)| = |\bar{\rho}(I_3)| = 3$. 
Furthermore, condition (d) and Lemma \ref{lem:lift_dihedral_even}(ii) imply that$|\chi(I_q)| = 2^k$. 
Therefore, we have $\mathrm{ord}(\chi) = 3 \times 2^k$. 
\end{proof}

 \begin{rem}
     By checking LMFDB \cite[Newform orbit 2601.1.x.a]{LMFDB}, one can find an example of a strongly minimal newform of $A_4$-type whose nebentypus has order $48$.
     This example coincides with the one constructed in Proposition \ref{prop:exicentece_A4_minimal}. 
     
     An $A_4$-type strongly minimal newform whose nebentypus has order $24$ cannot be found in LMFDB (\cite{LMFDB}).    
However, one can check its existence explicitely as follow: 
Let us take $q := 89$, which is the second smallest prime satisfying  $q \equiv \pm 1 \pmod{9}$ and $q \equiv 9 \pmod{16}$. 
Let $a \in \RR$ be one of the roots of the equation $X^3-3X -1 = 0$. 
Then we have $K = \QQ(a)$ and $\cO_K = \ZZ + \ZZ a + \ZZ a^2$ (see, for example, \cite[Number field 3.3.81.1]{LMFDB}). 
Since 
\[
12^3 - 3 \times 12 - 1 = 1691 = 19 \times 89,  
\]
we have a canonical isomorphism 
\[
\cO_K/(a-12)\cO_K = 
\ZZ/1691\ZZ \stackrel{\sim}{\to}
\ZZ/19\ZZ \times \ZZ/89\ZZ. 
\]
Moreover, $a + 1 \in \cO_K^\times$ and satisfies  
\begin{align*}
a+1 &\equiv 13 \pmod{(a-12)\cO_K}.
\end{align*}
Since $13$ is not a quadratic residue modulo $89$,  it follows from the proofs  of Lemma \ref{lem:A4-extension} and Proposition \ref{prop:exicentece_A4_minimal} that the existence of  an $A_4$-type strongly minimal newform with nebentypus of order $24$ and level $3^2 \times 89^2 = 71289$ is guaranteed. 
 \end{rem}

\subsubsection{$S_4$-type case}

Let $K_1/\QQ$ be the field defined by the polynomial $X^6 - 3X^4 - 5X^3 + 52X^2 - 92X + 56$ and 
let $K_2/\QQ$ be the field defined by the polynomial $X^6 - X^5 - 23X^4 - 45X^3 + 241X^2 + 1109X + 1177$. 
First, we list the properties of $K_1$ and $K_2$ that will be used below (see, for example, \cite[Number field 6.0.7880599.1 and Number field 6.0.1172648743.1]{LMFDB}: 
\begin{itemize}
    \item[(A)] The extension $K_i/\QQ$ is Galois with $\gal(K_i/\QQ) \simeq S_3$. 
    \item[(B)] The class number of $K_i$ is equal to $3$.
    \item[(C-1)] The primes ramified in $K_1$ are precisely $199$. Moreover, there are exactly three primes of $K_1$ lying above $199$.
\item[(C-2)]
The primes ramified in $K_2$ are precisely $7$ and $43$. Moreover, there are exactly three primes of $K_2$ lying above $7$, and exactly two primes of $K_2$ lying above $43$.
\item[(D)] There are exactly two primes of $K_i$ lying above $2$.
\end{itemize}
Note that, for any group isomorphism $\phi \colon S_3 \stackrel{\sim}{\to} \mathrm{Aut}(V_4)$, the semi-direct product $V_4 \rtimes_{\phi} S_3$ is isomorphic to $S_4$. 
We will construct an $S_4$-extension of $\QQ$ by building a $V_4$-extension of the $S_3$-extension $K_i/\QQ$.
Let $\sigma \in \Gal(K_i/\bQ) \simeq S_3$ be an element of order $3$.

\begin{lem}\label{lem:pm_square_existence_S4}
Let $i \in \{1,2\}$. 
    For any element $r \in \cO_{K_i}$ which is coprime to $2\cO_{K_i}$, 
     the congruence $X^2 \equiv  r r^\sigma \pmod{4\cO_{K_i}}$  admits a solution in $\cO_{K_i}$.  
\end{lem}
\begin{proof}
By property (D), there exist two distinct prime ideals $I_1, I_2 \subset \cO_{K_i}$ such that $2\cO_{K_i} = I_1I_2$. 
Since the order of $\sigma$ is $3$, we have $I_j^\sigma = I_j$. 
Hence  the solvability of $X^2 \equiv r \pmod{I_j^2}$ is equivalent to that of $X^2 \equiv r^\sigma \pmod{I_j^2}$. 
Since $(\cO_K/I_j^2)^\times \otimes_{\bZ} \bF_2 \simeq \bF_2$, 
the congruence $X^2 \equiv  r r^\sigma \pmod{I_j^2}$  admits a solution. 
\end{proof}

\begin{lem}\label{lem:S_4-units}
For each integer $i \in \{1,2\}$, the following exact sequence of $\bZ[\Gal(K_i/\bQ)]$-modules is split: 
\[
0 \to \{\pm 1\} \to \cO_{K_i}^\times 
\to \cO_{K_i}^\times/\{\pm 1\} \to 0. 
\]
\end{lem}
\begin{proof}
    Let $F/\QQ$ denote the quadratic extension satisfying $F \subset K_i$. 
Then $F = \QQ(\sqrt{-199})$ if $i=1$ and  $F = \QQ(\sqrt{-7})$ if $i=2$. 
In particular, $\cO_F^\times = \{\pm 1\}$. 
Since $[K_i \colon F] = 3$,  the norm map $\mathrm{Norm}_{K_i/F} \colon \cO_{K_i}^\times \to \cO_F^\times = \{\pm 1\}$  gives the splliting of the injection $\{\pm 1\} \hookrightarrow \cO_{K_i}^\times$. 
\end{proof}

Lemma \ref{lem:S_4-units}, together with Dirichlet's unit theorem, implies that one can choose elements $\epsilon_1, \epsilon_2 \in \cO_{K_i}^\times$  such that 
\[
 \epsilon_1^\sigma = \epsilon_2, \quad  \epsilon_2^\sigma = \epsilon_1 \epsilon_2, \quad  \textrm{$\{1, \epsilon_1, \epsilon_2, \epsilon_1\epsilon_2 \}$ is an $S_3$-module}. 
\]
Let $\tau$ denote the non-trivial element in $\gal(K_i/\QQ)$ satisfying $\epsilon_1^\tau = \epsilon_1$. Note that $\tau$ has order $2$ and satisfies $\tau \sigma = \sigma^2 \tau$. 


\begin{lem}\label{lem:generator_S4}
Let $i \in \{1,2\}$. 
    For any principal ideal $I \subset \cO_{K_i}$ with $I^\tau = I$, there exists an element $s \in \cO_{K_i}$ such that $s^\tau = \pm s$ and $s\cO_{K_i} = I$. 
\end{lem}
\begin{proof}
Dirichlet's unit theorem (see \cite[Proposition 8.7.2]{NSW08}) shows that there is an isomorphism $\cO_{K_i}^\times/\{\pm 1\} \simeq \ZZ^2$ of $\Gal(K_i/\bQ)$-modules  such that the automorphism $\tau$ acts on $\ZZ^2$ by $\begin{pmatrix}0&1\\1&0\end{pmatrix}$. 
This fact implies that the first cohomology group $H^1(\langle \tau \rangle, \cO_{K_i}^\times/\{\pm 1\})$ vanishes.
It therefore follows from Lemma~\ref{lem:S_4-units} that the injection $\{\pm 1\} \hookrightarrow \cO_{K_i}^\times$ induces an isomorphism
\[
H^1(\langle \tau \rangle, \{\pm 1\}) \xrightarrow{\ \sim\ }
H^1(\langle \tau \rangle, \cO_{K_i}^\times).
\]
Consequently, the exact sequence of $\Gal(K_i/\bQ)$-modules
\[
0 \to \cO_{K_i}^\times \to K_i^\times \to K_i^\times/\cO_{K_i}^\times \to 0
\]
gives rise to an exact sequence
\[
0 \to (\cO_{K_i}^\times)^{\tau = 1}
\to (K_i^\times)^{\tau = 1}
\to (K_i^\times/\cO_{K_i}^\times)^{\tau = 1}
\to H^1(\langle \tau \rangle, \cO_{K_i}^\times)
= H^1(\langle \tau \rangle, \{\pm 1\}), 
\]
which proves this lemma. 
\end{proof}

\begin{lem}\label{lem:S4-extension_1}
For any positive  integer $k$,  there exist infinitely many  primes $q \neq 199$ with  $\mathrm{ord}_2(q-1) = k$ for which an odd projective Galois representation $\bar{\rho} \colon G_\QQ \to \pgl_2(\CC)$  exists satisfying the following properties:
    \begin{itemize}
    \item[(a)] $\mathrm{Gal}(L/\QQ) \simeq S_4$, 
    \item[(b)] $\bar{\rho}$ is unramified outside $199$, $q$, and $\infty$, 
    \item[(c)] $\bar{\rho}(G_{\QQ_{199}}) \simeq \ZZ/2\ZZ$ or $D_4$, 
    \item[(d)] $\bar{\rho}(G_{\QQ_q}) \simeq D_4$. 
    \end{itemize}
\end{lem}
\begin{proof}

Consider the field 
\[
M := K_1(\sqrt{\epsilon_1}, \sqrt{\epsilon_2}, \zeta_{2^{k+1}}). 
\]
Since $\gal(K_1(\sqrt{\epsilon_1}, \sqrt{\epsilon_2})/\QQ) \simeq S_4$ by the choice of the units $\epsilon_1$ and $\epsilon_2$, 
we have $K_1(\sqrt{\epsilon_1}, \sqrt{\epsilon_2}) \cap K_1(\zeta_{2^{k+1}}) = K_1$, so that
\begin{align*}
    \gal(M/\QQ) &\simeq \gal(K_1(\sqrt{\epsilon_1}, \sqrt{\epsilon_2})/\QQ) \times \gal(\QQ(\zeta_{2^{k+1}})/\QQ)
    \\
    &\simeq S_4 \times (\ZZ/2^{k+1}\ZZ)^\times. 
\end{align*}
Therefore, by the Chebotarev density theorem, there are infinitely many  odd primes $q \neq 199$ such that the Frobenius conjugacy class at $q$ contains the element 
\[
(\begin{pmatrix}
    1&2&3&4
\end{pmatrix},  1 + 2^k)\in S_4 \times (\ZZ/2^{k+1}\ZZ)^\times
\]
under the above isomorphism. 
By construction, $q \equiv 1 + 2^k \pmod{2^{k+1}}$, which implies that  $\mathrm{ord}_2(q-1) = k$. 

Since $\begin{pmatrix}
    1&2&3&4
\end{pmatrix} \not\in V_4$,  the choice of the prime $q$ implies that there are exactly three primes $\mathfrak{q}_1, \mathfrak{q}_2, \mathfrak{q}_3$ of $K_1$ lying above $q$.
We may assume that $\mathfrak{q}_1^\sigma = \mathfrak{q}_2$ and $\mathfrak{q}_1^\tau = \mathfrak{q}_1$.

Since $\mathfrak{q}_1^\tau = \mathfrak{q}_1$, solvability of $X^2 \equiv \epsilon_2 \pmod{\mathfrak{q}_1}$ in $\cO_{K_1}$ is equivalent to that of
$X^2 \equiv \epsilon_2^\tau \pmod{\mathfrak{q}_1}$. 
Hence, the identity
\[
\epsilon_2 \epsilon_2^\tau = \epsilon_2 \epsilon_1^{\sigma\tau} = \epsilon_2 \epsilon_1^{\tau\sigma^2} = \epsilon_1 \epsilon_2^2
\]
implies that the congruence
\[
X^2 \equiv \epsilon_1 \pmod{\mathfrak{q}_1}
\]
admits a solution in $\cO_{K_1}$.
Since, by our choice of the rational prime $q$, the prime $\mathfrak{q}_1$ does not split
completely in $K_1(\sqrt{\epsilon_1}, \sqrt{\epsilon_2})$, the congruence
\[
X^2 \equiv \epsilon_2 \pmod{\mathfrak{q}_1}
\]
has no solution in $\cO_{K_1}$.

By property (B) and Lemma \ref{lem:generator_S4}, there exists an element $Q \in \cO_{K_1}$ satisfying $Q\cO_K = \mathfrak{q}_1^3$ and $Q^\tau = \pm Q$. 
We then have
\[
\pm q^3 = Q\, Q^\sigma\, Q^{\sigma^2}.
\]
By replacing $Q$ with $\epsilon_1 Q$, the element $Q^\sigma$ is replaced by $\epsilon_2 Q^\sigma$. Hence we may assume that the
congruence
\[
X^2 \equiv Q^\sigma \pmod{\mathfrak{q}_1}
\]
admits a solution in $\cO_{K_1}$. 
Using the relation $Q^{\sigma\tau} = Q^{\sigma^2}$, we obtain
\[
Q Q^\sigma Q^{\sigma^2}
= \pm q^3
= (Q Q^\sigma Q^{\sigma^2})^{\tau}
= \pm Q Q^{\sigma^2} Q^{\sigma},
\]
which implies that $Q^{\tau} = Q$.
It therefore follows from the relation $\mathfrak{q}_1^\tau = \mathfrak{q}_1$ that the congruence
\[
X^2 \equiv Q^\sigma Q^{\sigma^2} \pmod{\mathfrak{q}_1}
\]
admits a solution in $\cO_{K_1}$. 
Finally, by Lemma \ref{lem:pm_square_existence_S4}, the congruence
$X^2 \equiv QQ^\sigma \pmod{4\cO_{K_1}}$ also has a solution in $\cO_{K_1}$.


    Under the above preparations, set $L := K_1(\sqrt{QQ^{\sigma}}, \sqrt{Q^\sigma Q^{\sigma^2}})$. 
    Since
    \[
    \sqrt{QQ^{\sigma}}\sqrt{Q^{\sigma}Q^{\sigma^2}}\sqrt{Q^{\sigma^2}Q} = q^3, \quad Q^\tau = \pm Q, 
    \]
    the extension $L/\QQ$ is Galois, and there is an isomorphism $\mathrm{Gal}(L/\QQ) \simeq S_4$. 
Hence we obtain the projective Galois representation 
\[
\bar{\rho} \colon G_\QQ \to \gal(L/\QQ) \simeq S_4 \hookrightarrow \pgl_2(\CC). 
\]
Since $K_1$ is totally imaginary, the projective Galois representation  $\bar{\rho}$ is odd.

Since the congruence $X^2 \equiv QQ^{\sigma} \pmod{4\cO_{K_{1}}}$ has a solution, the primes dividing $2\cO_K$ are unramified in $L$.
Hence by construction, condition (b) is satisfied. 
Condition (c) follows from the property (C-1). 
Since the congruence $X^2 \equiv Q^\sigma \pmod{\mathfrak{q}_1}$ admits a solution, we have
\[
Q Q^{\sigma^2}
= \pm q^3/Q^{\sigma}
\in q \cdot (K_{1,\mathfrak{q}_1}^\times)^2.
\]
Therefore,
\[
K_{1,\mathfrak{q}_1}\bigl(\sqrt{Q Q^{\sigma}}, \sqrt{Q^\sigma Q^{\sigma^2}}\bigr)
= K_{1,\mathfrak{q}_1}(\sqrt{q}).
\]
Since $ K_{1,\mathfrak{q}_1}/\bQ_q$ is the unramified quadratic extension, this fact implies condition~(d). 
\end{proof}

\begin{lem}\label{lem:S4-extension_2}
For any positive  integer $k$,  there exist infinitely many  primes $q \neq 7, 43$ with  $\mathrm{ord}_2(q-1) = k$ for which an odd projective Galois representation $\bar{\rho} \colon G_\QQ \to \pgl_2(\CC)$  exists satisfying the following properties:
    \begin{itemize}
    \item[(a)] $\bar{\rho}(G_\QQ) \simeq S_4$, 
    \item[(b)] $\bar{\rho}$ is unramified outside $7$, $43$, $q$, and $\infty$, 
    \item[(c)] $\bar{\rho}(G_{\QQ_{43}}) = \bar{\rho}(I_{43}) \simeq \ZZ/3\ZZ$, 
    \item[(d)] $\bar{\rho}(G_{\QQ_q}) \simeq D_4$. 
    \end{itemize}
\end{lem}
\begin{proof}
By applying exactly the same method as in the proof of Lemma \ref{lem:S4-extension_1}, with $K_1$ replaced by $K_2$, we obtain an odd projective Galois representation
\[
\bar{\rho} \colon G_\QQ \to \mathrm{PGL}_2(\mathbb{C}),
\]
satisfying conditions (a), (b), and (d).
Among the subgroups of $S_4$, only those isomorphic to $\mathbb{Z}/3\mathbb{Z}$ admit a surjection onto $\mathbb{Z}/3\mathbb{Z}$. Therefore, property (C-2) implies condition (c). 
\end{proof}

\begin{prop}\label{prop:exicentece_S4_minimal_2^k}
    For any positive integer $k$, there exists a strongly minimal newform of $S_4$-type with nebentypus of order $2^k$. 
\end{prop}
\begin{proof}
By Lemma \ref{lem:S4-extension_1}, there exists an odd prime $q \neq 199$ with $\mathrm{ord}_2(q-1) = k$ and an odd projective Galois representat $\bar{\rho} \colon G_\QQ \to \pgl_2(\CC)$  satisfying the following properties:
    \begin{itemize}
    \item[(a)] $\mathrm{Gal}(L/\QQ) \simeq S_4$, 
    \item[(b)] $\bar{\rho}$ is unramified outside $199$, $q$, and $\infty$, 
    \item[(c)] $\bar{\rho}(G_{\QQ_{199}}) \simeq \ZZ/2\ZZ$ or $D_4$, 
    \item[(d)] $\bar{\rho}(G_{\QQ_q}) \simeq D_4$. 
    \end{itemize}
Since $\bar{\rho}(G_\QQ) \simeq S_4$, it follows from the theorem of Langlands--Tunnel \cite{Tunnell81} that  one can choose a strongly minimal newform $f$ satisfying $\bar{\rho}_f \simeq \bar{\rho}$. 
From Lemma \ref{lem:order_character} and condition (b), we obtain 
\[
\mathrm{ord}(\chi) = \lcm\{|\chi(I_{199})|, |\chi(I_q)|\}. 
\]
 Lemma \ref{lem:lift_cyclic} or Lemma \ref{lem:lift_dihedral_even}(ii), combined with condition (c), implies $|\chi(I_{199})| = 2$. 
Furthermore, condition (d) and Lemma \ref{lem:lift_dihedral_even}(ii) imply that$|\chi(I_q)| = 2^k$. 
Therefore, we have $\mathrm{ord}(\chi) =  2^k$. 
\end{proof}

\begin{prop}\label{prop:exicentece_S4_minimal_3*2^k}
    For any positive integer $k$, there exists a strongly minimal newform of $S_4$-type with nebentypus of order $3 \times 2^k$. 
\end{prop}
\begin{proof}
By Lemma \ref{lem:S4-extension_2}, there exists an odd prime $q \neq 7, 43$ with $\mathrm{ord}_2(q-1) = k$ and an odd projective Galois representat $\bar{\rho} \colon G_\QQ \to \pgl_2(\CC)$  satisfying the following properties:
    \begin{itemize}
    \item[(a)] $\bar{\rho}(G_\QQ) \simeq S_4$, 
    \item[(b)] $\bar{\rho}$ is unramified outside $7$, $43$, $q$, and $\infty$, 
    \item[(c)] $\bar{\rho}(G_{\QQ_{43}}) = \bar{\rho}(I_{43}) \simeq \ZZ/3\ZZ$, 
    \item[(d)] $\bar{\rho}(G_{\QQ_q}) \simeq D_4$. 
    \end{itemize}
Since $\bar{\rho}(G_\QQ) \simeq S_4$, it follows from the theorem of Langlands--Tunnel \cite{Tunnell81} that one can choose a strongly minimal newform $f$ satisfying $\bar{\rho}_f \simeq \bar{\rho}$. 
From Lemma \ref{lem:order_character} and condition (b), we obtain 
\[
\mathrm{ord}(\chi) = \lcm\{|\chi(I_{7})|, |\chi(I_{43})|, |\chi(I_q)|\}. 
\]
It  follows from condition (c) and Lemma \ref{lem:lift_cyclic}  that $|\chi(I_{43})| = |\bar{\rho}(I_{43})| = 3$. 
Furthermore, condition (d) and Lemma \ref{lem:lift_dihedral_even}(ii) together imply that$|\chi(I_q)| = 2^k$. 
Since $\chi$ is tamely ramified at $7$, the order $|\chi(I_7)|$ is divisible by $6$. 
Therefore, we have $\mathrm{ord}(\chi) = 3 \times 2^k$. 
\end{proof}

\begin{rem}
    As shown in Theorem \ref{thm:mainS4general}, when $\mathrm{ord}_2(d) \leq 2$, there are two possible Hecke fields for newforms of $S_4$-type. 
     In the above construction, it is not clear whether there exist strongly minimal newforms realizing each of these Hecke fields. 
     However, by checking LMFDB (\cite{LMFDB}), one can easily see that when $d \in \{2,4,6,12\}$, all of the candidates actually occur as the Hecke fields of strongly minimal newforms.
\end{rem}

\section{The group structure of the image of $\rho_f$}
\label{sec:group_structure_of_the_image_of_rho_f}
In this section, we continue to let $f$ be a weight one exotic newform with nebentypus $\chi$, and let $d$ denote the order of $\chi$.
This section is devoted to determining the image of the Galois representation $\rho_f \colon G_\bQ \to \mathrm{GL}_2(\bC)$ associated with $f$, in terms of the type of $f$ and the order $d$ of $\chi$. 
In the case of square-free level, an analogous result was obtained by Kida and Sudo \cite{kida2017two}.

\subsection{Preliminaries}
\subsubsection{Settings}
Let $G := \rho_f(G_\bQ) \subset \gl_2(\CC)$ denote the image of $\rho_f$. Write $K := G\cap \operatorname{SL}_2(\CC) = \ker (\det|_G)$.
Also, let 
$\pi\colon \gl_2(\CC)\twoheadrightarrow \pgl_2(\CC)$
denote the natural projection, and set
$\overline{G}:=\pi(G)$ and  $\overline{K}:=\pi(K)$; i.e., the projective images of $G$ and $K$, respectively. 

We write $I := \begin{pmatrix}1&0\\0&1\end{pmatrix}$ for the $2$ by $2$ identity matrix. 
Since  $\ker (\pi|_G)=G\cap \CC^\times I$ is a finite subgroup of $\CC^\times$, it is of the form $\mu_a I$ for some integer $a\geq 1$, where $\mu_a\subset\CC^\times$ denotes the group of $a$-th roots of unity. 
Since $\det \colon G/K \stackrel{\sim}{\to} \mu_d$ is cyclic,   $\overline{G}/\overline{K}$ is also cyclic. 
Finally, we remark on the following lemma; it follows from the fact that the groups $A_4$, $S_4$, and $A_5$ are centerless.

\begin{lem}
    The center of the group $G$ is $\mu_a I$. 
\end{lem}

\subsubsection{The structure of the group $K$}

We begin by recalling the classical result of Klein on the classification of finite subgroups of $\operatorname{SL}_2(\CC)$ (see, for example, \cite{Flicker}). 
 
\begin{prop}
\label{prop:SL2classification}
The finite subgroups of $\operatorname{SL}_2(\CC)$ are isomorphic to one of the following (the indices indicate the orders of the groups):
\begin{itemize}
    \item the cyclic group $\ZZ/n\ZZ$,
    \item the binary dihedral group $BD_{4n}$,
    \item the binary tetrahedral group $BT_{24} \simeq \operatorname{SL}_2(\FF_3)$, which is the double cover of $A_4$, 
    \item the binary octahedral group $BO_{48}$, which is a double cover of $S_4$, 
    \item the binary icosahedral group $BI_{120} \simeq \operatorname{SL}_2(\FF_5)$, which is the double cover of $A_5$.  
\end{itemize}
\end{prop}


\begin{rem}\label{rem:double_cover_S_4}
The symmetric group $S_4$ of degree four has two double covers: one is $BO_{48}$, and the other is $\GL_2(\bF_3)$ (note that $\PGL_2(\bF_3) \simeq \mathrm{Aut}(\bP^1(\bF_3)) \simeq S_4$).
The binary octahedral group $BO_{48}$ has the presentation
\[
\langle a,b,c \mid a^4 = b^3 = c^2 = abc \rangle 
\] 
and  $BO_{48}/\langle abc \rangle \simeq  S_4$. 
The group $BO_{48}$ has no subgroup of order $2$ other than its center $\langle abc \rangle$, whereas the group $\GL_2(\bF_3)$ possesses a non-normal subgroup of order $2$. 
This property allows one to distinguish between the two groups. 
\end{rem}

\begin{prop}
\label{prop:Kbar}
    We have the following cases:
    \[
    \overline{K} \simeq \begin{cases}
        V_4 \ \ \textrm{or} \ \ A_4 & \textrm{ if } \,\, \overline{G}\simeq A_4, 
        \\
        A_4 \ \ \textrm{or} \ \ S_4 & \textrm{ if } \,\, \overline{G}\simeq S_4,\\
        A_5  & \textrm{ if } \,\, \overline{G}\simeq A_5,
    \end{cases}
    \] 
    where $V_4 \subset A_4$ denotes the Klein four-group, which is isomorphic to $\ZZ/2\ZZ \times \bZ/2\bZ$.
    In particular, 
    \[
    |\overline{G}/\overline{K}| = \begin{cases}
        3 \ \ \textrm{or} \ \ 1 & \textrm{ if } \,\, \overline{G}\simeq A_4, 
        \\
        2 \ \ \textrm{or} \ \ 1 & \textrm{ if } \,\, \overline{G}\simeq S_4,\\
        1  & \textrm{ if } \,\, \overline{G}\simeq A_5.
    \end{cases}
    \] 
\end{prop}
\begin{proof}
Since $\overline{G}$ is non-commutative and $\overline{G}/\overline{K}$ is cyclic, $\overline{K}$ is a non-trivial normal subgroup of $\overline{G}$, and this fact implies the assertion.
\end{proof}

\begin{cor}
\label{cor:Kclassification}
    We have the following cases:
    \[
    K \simeq \begin{cases}
        Q_8 & \textrm{ if } \,\, \overline{G}\simeq A_4 \, \, \textrm{ and } \,\, \overline{K} \neq \overline{G},
        \\
        \SL_2(\bF_3) & \textrm{ if } \,\, \overline{G}\simeq A_4 \, \, \textrm{ and } \,\, \overline{K} = \overline{G}, 
        \\
        \SL_2(\bF_3) & \textrm{ if } \,\, \overline{G}\simeq S_4 \, \, \textrm{ and } \,\, \overline{K} \neq \overline{G}, 
        \\
        BO_{48} & \textrm{ if } \,\, \overline{G}\simeq S_4 \, \, \textrm{ and } \,\, \overline{K} = \overline{G}, 
        \\
        \SL_2(\bF_5)  & \textrm{ if } \,\, \overline{G}\simeq A_5.
    \end{cases}
    \]
\end{cor}
\begin{proof}
Since the  binary dihedral $BD_8$ of order $8$ is isomorphic to the the quaternion group $Q_8$, this corollary follows from  Propositions \ref{prop:SL2classification} and \ref{prop:Kbar}.
\end{proof}


\begin{cor}
\label{cor:order2}
    We have $-I\in K$. In particular, $K\cap \mu_aI=\{\pm I\}$, the integer $a$ is even, and the group $\overline{G}/\overline{K}$ is cyclic of order $2d/a$.
\end{cor}
\begin{proof}
Note that $-I$ is the unique element of order $2$ in $\operatorname{SL}_2(\CC)$. 
Hence, any finite subgroup of $\operatorname{SL}_2(\CC)$ of even order contains $-I$. 
The remaining assertions are then clear (see the commutative diagram \eqref{eq:nine_exact_sequence} below). 
\end{proof}

In what follows, for each integer $n \geq 1$, the cyclic group $\ZZ/n\ZZ$ is identified with $\mu_n \subset \CC^\times$ via the homomorphism $k \bmod n \mapsto \exp(2\pi \sqrt{-1} k / n)$. 
Then the situation is summarized as the following commutative diagram, all of whose horizontal and vertical lines are exact: 
\begin{align}\label{eq:nine_exact_sequence}
\begin{split}
        \xymatrix{
  & 1 \ar[d] & 1 \ar[d] & 1 \ar[d] & \\
 1 \ar[r] & \bZ/2\bZ \ar[d] \ar[r] & \bZ/a\bZ \ar[d] \ar[r]^-{\det} 
      & \bZ/(a/2)\bZ \ar[d] \ar[r] & 1 \\
  1 \ar[r] & K \ar[d]_-\pi \ar[r] & G \ar[d]_-\pi \ar[r]^-{\det} & \bZ/d\bZ \ar[d] \ar[r] & 1 \\
1 \ar[r]  & \overline{K} \ar[d] \ar[r] & \overline{G} \ar[d] \ar[r] & \overline{G}/\overline{K} \ar[r] \ar[d] & 1  \\
  & 1 & 1 & 1 &
}
\end{split}
\end{align}

\subsection{Main results on the image of $\rho_f$}
In this subsection, we explain the main results of \S\ref{sec:group_structure_of_the_image_of_rho_f}. 
Recall that $\pi \colon \mathrm{GL}_2(\CC) \twoheadrightarrow \mathrm{PGL}_2(\CC)$ is the projection map, and that
\[
\overline{G} = \pi(\rho_f(G_\QQ)) \quad \text{and} \quad \overline{K} = \pi(\rho_f(G_\QQ) \cap \mathrm{SL}_2(\CC)).
\]
Also recall that $d$ is the order of $\chi$, and that $I = \begin{pmatrix}1 & 0 \\ 0 & 1\end{pmatrix}$ denotes the $2 \times 2$ identity matrix.

\begin{thm}\label{thm:Galois_image_A_4_type}
Suppose that $f$ is of $A_4$-type. 
\begin{itemize}
    \item[(i)] If $\overline{K} = \overline{G}$, then the Galois image $\rho_f(G_\QQ)$ is isomorphic to the group 
    \[
    (\SL_2(\bF_3) \times \ZZ/2d\ZZ)/\langle (-I, d) \rangle.
    \]
    \item[(ii)] If $\overline{K} \neq \overline{G}$, then the integer $d$ is divisible by $3$, and the Galois image $\rho_f(G_\QQ)$ is isomorphic to the group 
    \[
    (\SL_2(\bF_3) \times_{\ZZ/3\ZZ} \ZZ/2d\ZZ)/\langle (-I, d) \rangle.
    \]
Here, the homomorphism $\SL_2(\FF_3) \twoheadrightarrow \ZZ/3\ZZ$ used in the fiber product is defined as the composition
\[
\SL_2(\FF_3) \twoheadrightarrow \SL_2(\FF_3)^{\mathrm{ab}}  \simeq \ZZ/3\ZZ.
\]
\end{itemize}
\end{thm}
\begin{proof}
        Claim (i) follows from Corollary \ref{cor:Kclassification} and Proposition \ref{prop:proof_of_thm_when_K=G} below. 
        Claim (ii) follows from Proposition \ref{prop:A4_type_image} below. 
\end{proof}

\begin{thm}\label{thm:Galois_image_A_5_type}
If $f$ is of $A_5$-type, then the Galois image $\rho_f(G_\QQ)$  is isomorphic to the group 
    \[
    (\SL_2(\bF_5) \times \ZZ/2d\ZZ)/\langle (-I, d) \rangle.
    \]
\end{thm}
\begin{proof}
    This theorem follows from Corollary \ref{cor:Kclassification} and Proposition \ref{prop:proof_of_thm_when_K=G} below. 
\end{proof}

\begin{thm}\label{thm:Galois_image_S_4_type}
Suppose that $f$ is of $S_4$-type. 
Let $BO_{48}$ be the binary dihedral group of order $48$ and 
$z \in BO_{48}$ denotes the unique element of order $2$ (see Remark \ref{rem:double_cover_S_4}). 
\begin{itemize}
    \item[(i)] If $\overline{K} = \overline{G}$, then the Galois image $\rho_f(G_\QQ)$  is isomorphic to the group 
    \[
    (BO_{48} \times \ZZ/2d\ZZ)/\langle (z, d) \rangle. 
    \]
    \item[(ii)] If $\overline{K} \neq \overline{G}$, then the Galois image $\rho_f(G_\QQ)$  is isomorphic to the group 
    \[
(BO_{48} \times_{\bZ/2\bZ} \ZZ/2d\ZZ)/\langle (z, d) \rangle.
    \]
    Here, the homomorphism $BO_{48} \twoheadrightarrow \ZZ/2\ZZ$ used in the fiber product is defined as the composition
\[
BO_{48} \twoheadrightarrow BO_{48}^{\rm ab} \simeq   \ZZ/2\ZZ.
\]
\end{itemize}
\end{thm}
\begin{proof}
        Claim (i) follows from Corollary \ref{cor:Kclassification} and Proposition \ref{prop:proof_of_thm_when_K=G} below. 
        Claim (ii) follows from Proposition \ref{prop:S4_type_image} below. 
\end{proof}

\subsection{Criterion for determining whether $\overline{K} = \overline{G}$ or not}

When $f$ is of $A_4$- or $S_4$-type, it may happen that $\overline{K} \neq \overline{G}$. 
In this subsection, we give a necessary and sufficient condition for the inequality $\overline{K} \neq \overline{G}$ to occur. 
We also briefly discuss whether this condition can be effectively checked (under the Generalized Riemann Hypothesis). 




\subsubsection{$A_4$-case}
\begin{prop}
\label{prop:A4_criteriaon_K_neq_G}
    Suppose that $\overline{G}\simeq A_4$. Let $\varphi \colon G \twoheadrightarrow \bZ/3\bZ$ be the composite homomorphism
\[
\varphi \colon G \stackrel{\pi}{\twoheadrightarrow} \overline{G} \simeq A_4 \twoheadrightarrow \bZ/3\bZ. 
\]
Then $\overline{K} \neq \overline{G}$ if and only if $\ker\left(\mathrm{det} \colon G \to \bZ/d\bZ \right) \subset \ker(\varphi)$. In particular, if $3\nmid d$, then we have $\overline{K}=\overline{G}$ and $a=2d$. 
\end{prop}
\begin{proof}
Assume first that $\overline{K}\neq \overline{G}$; that is, $\overline{K}\simeq V_4$ by Proposition \ref{prop:Kbar}. In this case, the composite $G \stackrel{\pi}{\twoheadrightarrow} \overline{G} \twoheadrightarrow \overline{G}/\overline{K} \simeq \bZ/3\bZ$ is identified with the  homomorphism $\varphi$.
Since the homomorphism  $G \twoheadrightarrow \overline{G}/\overline{K}$ factors through $G \stackrel{\det}{\twoheadrightarrow} G/K$, we obtain $\ker(\mathrm{det}) \subset \ker(\varphi)$.  
When $\overline{K} = \overline{G}$, we have 
$\varphi(K) = \varphi(G) = \bZ/3\bZ$, since $\varphi$ factors through $\pi$ by definition. 
In particular, $\ker (\det) = K \not\subset \ker(\varphi)$. 
The remaining assertion follows from Corollary \ref{cor:order2}. 
\end{proof}

Next, we discuss whether the equivalent conditions for $\overline{K} \neq \overline{G}$ established in Proposition \ref{prop:A4_criteriaon_K_neq_G}  can be effectively verified.

Suppose that $\overline{G} \simeq A_4$ and  $3 \mid d$. 
We then have the two Dirichlet characters of order $3$
\[
\chi' \colon  G_\bQ \xrightarrow{\overline{\rho}_f}  \overline{G} \twoheadrightarrow \mu_3 \quad \text{and} \quad \mathrm{det}^{d/3} \circ \rho_f = \chi^{d/3}.
\]
Note that Lemma \ref{lem:BL_c(g)} shows that $\chi'(\mathrm{Frob}_p) = 1$ if and only if $a_p(f) = 0$ for any prime $p \nmid N$, where $N$ denotes the level of $f$. 
Since the homomorphism $\rho_f \colon G_\QQ \twoheadrightarrow G$ is  surjective by definition, it  follows from Proposition \ref{prop:A4_criteriaon_K_neq_G} that  $\overline{K} \neq \overline{G}$ if and only if  
\begin{align*}
\{p \mid \mathrm{gcd}(p,N) = 1 \, \textrm{ and } \, \chi^{d/3}(\mathrm{Frob}_p) = 1\} &= \{p \mid \mathrm{gcd}(p,N) = 1 \, \textrm{ and } \, \chi'(\mathrm{Frob}_p) = 1\} 
\\
&= \{p \mid \mathrm{gcd}(p,N) = 1 \, \textrm{ and } \, a_p(f) = 0\}.     
\end{align*}

\begin{lem}\label{lem:effective_chebotarev_order_3_dirichlet_character}
Let $L/\bQ$ be a Galois extension of degree $9$ unramified outside $N$. 
Let $\mathrm{rad}(N)$ denote the positive integer obtained as the product of the primes dividing $N$, that is,
\[
\mathrm{rad}(N) := \prod_{\ell \mid N} \ell.
\]
For any element $\sigma \in \mathrm{Gal}(L/\bQ)$,  there exists a prime $p \nmid N$  such that $p \leq (3^{18} \cdot \mathrm{rad}(N)^8 )^{12577}$ and $\mathrm{Frob}_p = \sigma$ in $\mathrm{Gal}(L/\bQ)$. 
Moreover, if we assume the Generalized Riemann Hypothesis, then one can replace the inequality $p \leq (3^{18} \cdot \mathrm{rad}(N)^8 )^{12577}$ with the inequality 
\[
p \leq 2^{10} \cdot (\mathrm{log} (\mathrm{rad}(N)) + 2)^2. 
\]
\end{lem}
\begin{proof}
Since $L/\bQ$ is a Galois extension of degree $9$ unramified outside $N$, its discriminant satisfies
\[
|\mathrm{disc}(L/\QQ)| \le 3^{18} \cdot \mathrm{rad}(N)^8.
\]
Hence, this lemma follows from effective versions of the Chebotarev density theorem proved by Ahn and Kwon  \cite[Theorem 1.1]{Jeoung-Hwan_Soun-Hi_effective_chebotarev} and by Bach and Sorenson   \cite[Theorem 5.1]{Bach_Soreson_1996}.
\end{proof}

\begin{prop}\label{prop:effective_version_K_neq_G}
Suppose that $\overline{G} \simeq A_4$ and that $3 \mid d$. 
Let 
\[
M := (3^{18} \cdot \mathrm{rad}(N)^8 )^{12577},
\]
where $N$ denotes the level of $f$. 
Then the following are equivalent:
\begin{itemize}
    \item[(a)] $\overline{K} \neq \overline{G}$.
    \item[(b)] $\ker\left(\chi \colon G_\bQ \twoheadrightarrow \mu_d \right) \subset \ker(G_\bQ \stackrel{\bar{\rho}_f}{\to} A_4 \twoheadrightarrow A_4^{\rm ab})$. 
    \item[(c)]  $\{\, p < M \mid (p,N)=1 \text{ and } \chi^{d/3}(\mathrm{Frob}_p)=1 \,\}
    =
    \{\, p < M \mid (p,N)=1 \text{ and } a_p(f)=0 \,\}$. 
\end{itemize}
Moreover, under the Generalized Riemann Hypothesis, one may replace $M$ by
\[
2^{10} \cdot (\mathrm{log} (\mathrm{rad}(N)) + 2)^2. 
\]
\end{prop}

\begin{rem}
Given a ($q$-expansion of a) weight one newform $f$ and a prime $p$, the values
$a_p(f)$ and $\chi(\mathrm{Frob}_p) = a_p(f)^2 - a_{p^2}(f)$ can be checked easily.
However, since $3^{18\times 12577}$ has roughly $108000$ digits., it is not
realistic to verify the equality in Proposition \ref{prop:effective_version_K_neq_G}(b) by brute-force
computation.
On the other hand, under  the Generalized Riemann
Hypothesis, checking the equality in Proposition \ref{prop:effective_version_K_neq_G}(b) is computationally feasible. 
\end{rem}


\begin{proof}
We have already shown that (a) and (b) are equivalent 
and that (a) implies (c), so it suffices to prove that (c) implies (a). 
Assume that $\overline{K} = \overline{G}$. 
Let $L$ be the number field corresponding to the open  subgroup
\[
\ker(\chi^{d/3}) \cap \ker(\chi') \subset G_\QQ.
\]
By Proposition~\ref{prop:A4_criteriaon_K_neq_G}, the extension $L/\QQ$ is abelian of degree $9$ and 
\[
(\chi^{d/3}, \chi') \colon  
\mathrm{Gal}(L/\bQ) \xrightarrow{\sim} \mu_3 \times \mu_3. 
\]
Hence Lemma \ref{lem:effective_chebotarev_order_3_dirichlet_character} implies that there exists a prime 
$p < M$ such that 
\[
\chi^{d/3}(\mathrm{Frob}_p) \neq 1 
\quad \text{and} \quad 
\chi'(\mathrm{Frob}_p) = 1.
\]
Since the condition $\chi'(\mathrm{Frob}_p)=1$ is equivalent to that $a_p(f)=0$, 
this shows that (c) does not hold, and hence completes the proof. 
\end{proof}

\subsubsection{$S_4$-case}
By the same argument as in Propositions \ref{prop:A4_criteriaon_K_neq_G} and \ref{prop:effective_version_K_neq_G}, a similar result is obtained in the $S_4$-case.

\begin{prop}
\label{prop:S4_criteriaon_K_neq_G}
Suppose that $\overline{G} \simeq S_4$. 
Let 
\[
M := (2^{24} \cdot 3^{12} \cdot \mathrm{rad}(N)^{11} )^{12577},
\]
where $N$ denotes the level of $f$. 
Then the following are equivalent:
\begin{itemize}
    \item[(a)] $\overline{K} = \overline{G}$.
    \item[(b)] $\mathrm{sgn}\circ \overline{\rho}_f \neq  \chi^{d/2}$ as Dirichlet characters. 
    \item[(c)]  There exists a prime $p \nmid N$ such that $p < M$ and $a_p(f)^{d} = -1$. 
\end{itemize}
Moreover, under the Generalized Riemann Hypothesis, one may replace $M$ by
\[
16 \cdot (11 \cdot \mathrm{log} (\mathrm{rad}(N)) + 12)^2. 
\]
\end{prop}
\begin{proof}
The equivalence between (a) and (b) follows by the same argument as in the proof of Proposition \ref{prop:A4_criteriaon_K_neq_G}. Therefore, it remains to show that (b) and (c) are equivalent. 
Since $\overline{G} \simeq S_4$, it follows from Lemma \ref{lem:BL_c(g)} that $|a_p(f)| = 1$
if and only if $\bar{\rho}_f(\mathrm{Frob}_p) \in S_4$ has order $3$, in which case
\[
a_p(f)^d = \chi^{d/2}(p).
\]
Thus, if $\mathrm{sgn}\circ \overline{\rho}_f = \chi^{d/2}$, then for any prime
$p \nmid N$ such that $\bar{\rho}_f(\mathrm{Frob}_p) \in S_4$ has order $3$, it follows that
\[
a_p(f)^d
= \chi^{d/2}(p)
= \mathrm{sgn}\circ \overline{\rho}_f(\mathrm{Frob}_p)
= 1.
\]
This proves that (c) implies (b). We now prove the converse.
Suppose that $\mathrm{sgn}\circ \overline{\rho}_f \neq \chi^{d/2}$.
Let $L_1$ denote the $S_3$-extension of $\bQ$ corresponding to the open subgroup
\[
\ker\bigl(G_\bQ \stackrel{\bar{\rho}_f}{\twoheadrightarrow} S_4 \twoheadrightarrow S_3\bigr) \subset G_\bQ,
\]
and let $L_2$ denote the quadratic extension of $\bQ$ corresponding to the quadratic character $\chi^{d/2}$. 
The assumption that $\mathrm{sgn}\circ \overline{\rho}_f \neq \chi^{d/2}$ then implies that
$L_1 \cap L_2 = \bQ$, and hence
\[
\mathrm{Gal}(L_1L_2/\bQ) \simeq \mathrm{Gal}(L_1/\bQ) \times \mathrm{Gal}(L_2/\bQ) \simeq S_3 \times \{\pm 1\}.
\]
The effective version of the Chebotarev density theorem proved by
Ahn and Kwon \cite[Theorem 1.1]{Jeoung-Hwan_Soun-Hi_effective_chebotarev}
(resp. by Bach and Sorenson \cite[Theorem 5.1]{Bach_Soreson_1996} under the
Generalized Riemann Hypothesis) implies that there exists a prime
$p \nmid N$ with $p < M$ (resp. $p < 16 \cdot (11 \cdot \mathrm{log} (\mathrm{rad}(N)) + 12)^2$) such that the Frobenius conjugacy
class at $p$ in $\mathrm{Gal}(L_1L_2/\bQ) \simeq S_3 \times \{\pm 1\}$ contains $(\begin{pmatrix}1&2&3\end{pmatrix}, -1)$. 
Then the element $\bar{\rho}_f(\mathrm{Frob}_p) \in S_4$ has order $3$, and hence $a_p(f)^d = \chi^{d/2}(p) = -1$. 
\end{proof}

\subsection{The structure of $G$ when $\overline{K}=\overline{G}$}

First, we determine the structure of $G$ under the assumption $\overline{K}=\overline{G}$.
\begin{prop}\label{prop:proof_of_thm_when_K=G}
     Assume that $\overline{K}=\overline{G}$.
     \begin{itemize}
     \item[(i)]  We have $a=2d$ and  $G=K\cdot \mu_{2d}I$.
    \item[(ii)] The group $G$ is isomorphic to the group 
    \[
 (K\times \ZZ/{2d}\ZZ)/\langle (-I, d) \rangle. 
    \]
     \end{itemize}
 \end{prop}
 \begin{proof}
It follows from the assumption $\overline{K} = \overline{G}$  that  $G = K \cdot \mu_{a} I$ and that $a = 2d$ by Corollary \ref{cor:order2}. 
Claim (ii) follows from Corollary \ref{cor:order2} together with claim (i) since $\mu_a I$ is the center of $G$. 
 \end{proof}

\subsection{The structure of $G$ when $\overline{K}\neq \overline{G}$} 
Let us consider the case $\overline{K}\neq \overline{G}$. In this case, Proposition \ref{prop:Kbar} implies that $f$ is either of $A_4$-type or $S_4$-type. We shall determine the structure of $G$ in this setting. 

Note that, by the commutative diagram \eqref{eq:nine_exact_sequence}, 
the group \(G\) fits into the central extension
\begin{align}\label{eq:central_extension_Z/a_G_Gbar}
1 \to \ZZ/a\ZZ \to G \stackrel{\pi}{\to} \overline{G} \to 1.
\end{align} 
Such central extensions are classified by the second cohomology group $H^2(\overline{G}, \ZZ/a\ZZ)$ (see, for example, \cite[Theorem 1.2.4]{NSW08} or \cite[Theorem 6.6.3]{Weibel94}). 
We first recall the universal coefficients theorem for group cohomology.
\begin{prop}
\label{prop:UCT}
    Let $H$ be a finite group and $A$ be a finite abelian group.
    Then there exists an exact sequence of abelian gruops: 
    \[
    0\to \operatorname{Ext}^1_{\ZZ}(H^{\mathrm{ab}}, A)\to H^2(H,A)\to \operatorname{Hom}(H_2(H,\ZZ),A)\to 0,
    \]
   which splits non-canonically.
\end{prop}
\begin{proof}
    See \cite[VI. Theorem 15.1]{Hilton-Stammbach97} or \cite[(3.6.5)]{Weibel94}, for example.
\end{proof}

\begin{rem}
The homomorphism $\operatorname{Ext}^1_{\bZ}(H^{\mathrm{ab}}, A) \to H^2(H, A)$ is defined by sending
each extension of abelian groups
\[
0 \to A \to M \to H^{\mathrm{ab}} \to 0
\]
to its pullback along the surjective homomorphism $H \twoheadrightarrow H^{\mathrm{ab}}$, namely,
\[
0 \to A \to M \times_{H^{\mathrm{ab}}} H \to H \to 0.
\]
\end{rem}

\subsubsection{$A_4$-case}

Suppose here that $(\overline{G}, \overline{K})=(A_4,V_4)$.
In this case, it follows from Corollary \ref{cor:order2} that 
\[
|\overline{G}/\overline{K}|=3,  \quad 3a=2d. 
\]
Since $A_4^{\rm ab}\simeq \ZZ/3\ZZ$ and $H_2(A_4,\ZZ)\simeq \ZZ/2\ZZ$  
(see, for example, \cite[(2.22), Chapter 3]{Suzuki82}), 
the exact sequence given in Proposition \ref{prop:UCT} with $H=A_4$ and $A=\ZZ/a\ZZ$ yields a split exact sequence of abelian grouops
 \[
 0\to \operatorname{Ext}^1_{\bZ}(\ZZ/3\ZZ, \ZZ/a\ZZ)\to H^2(A_4,\ZZ/a\ZZ)\to \operatorname{Hom}(\ZZ/2\ZZ,\ZZ/a\ZZ)\to 0.
 \]
Write $a = 3^b c$ with $3 \nmid c$ and $c \in 2\bZ$. 
Then the canonical isomorphism $\ZZ/a\ZZ \stackrel{\sim}{\to} \ZZ/3^b\ZZ \times \ZZ/c\ZZ$ induces the following identifications:
\begin{align*}
\Ext^1_{\ZZ}(\ZZ/3\ZZ, \ZZ/a\ZZ)
&\stackrel{\sim}{\to} \Ext^1_{\ZZ}(\ZZ/3\ZZ, \ZZ/3^b\ZZ)
   \oplus \Ext^1_{\ZZ}(\ZZ/3\ZZ, \ZZ/c\ZZ) \\
&= \Ext^1_{\ZZ}(\ZZ/3\ZZ, \ZZ/3^b\ZZ) \oplus \{0\}, \\
H^2(A_4,\ZZ/a\ZZ)
&\stackrel{\sim}{\to} H^2(A_4,\ZZ/3^b\ZZ) \oplus H^2(A_4,\ZZ/c\ZZ), \\
\Hom(\ZZ/2\ZZ,\ZZ/a\ZZ)
&\stackrel{\sim}{\to} \Hom(\ZZ/2\ZZ,\ZZ/3^b\ZZ)
   \oplus \Hom(\ZZ/2\ZZ,\ZZ/c\ZZ) \\
&= \{0\} \oplus \Hom(\ZZ/2\ZZ,\ZZ/c\ZZ).
\end{align*}
Consequently, we obtain the following lemma: 

\begin{lem}\label{lem:str_of_ext_H^2_A_4_case}
The group $H^2(A_4, \ZZ/2\ZZ)$ is isomorphic to $\ZZ/2\ZZ$,
and the nontrivial element in $H^2(A_4, \ZZ/2\ZZ)$ corresponds to the central extension
\[
1 \to \ZZ/2\ZZ \to \SL_2(\FF_3) \to A_4 \to 1.
\]
Moreover, the unique injective homomorphism $\ZZ/2\ZZ \hookrightarrow \ZZ/c\ZZ$ induces an isomorphism
\[
H^2(A_4, \ZZ/2\ZZ) \stackrel{\sim}{\to} H^2(A_4, \ZZ/c\ZZ).
\]
\end{lem}

For notational simplicity, we set
\[
G_c := G/\ker(\bZ/a\bZ \twoheadrightarrow \bZ/3^b\bZ), \qquad
G_3 := G/\ker(\bZ/a\bZ \twoheadrightarrow \bZ/c\bZ).
\]
We then obtain two central extensions
\begin{align*}
    1 \to \bZ/c\bZ \to G_c \to A_4 \to 1, \qquad
    1 \to \bZ/3^b\bZ \to G_3 \to A_4 \to 1,
\end{align*}
together with the canonical isomorphism 
\begin{align}\label{eq:A_4_case_isom_G_3_and_G_c}
G \stackrel{\sim}{\to} G_c \times_{A_4} G_3. 
\end{align}

\begin{lem}
\label{lem:central_extension_A_4_V_4}\ 
    \begin{itemize}
        \item[(i)]  We have an isomorphism of groups 
        \[
        G_c\simeq (\SL_2(\bF_3) \times \ZZ/c\ZZ)/\langle (-I, c/2)\rangle.
        \]
        Moreover, the homomorphism $G_c \twoheadrightarrow A_4$ corresponds to the first projection.  
        \item[(ii)]
   We have an isomorphism of groups 
        \[
        G_3\simeq A_4 \times_{\ZZ/3\ZZ} \ZZ/{3^{b+1}}\ZZ.
        \]
            Moreover, the homomorphism $G_3 \twoheadrightarrow A_4$ corresponds to the first projection.  
    \end{itemize}
\end{lem}



\begin{proof}
Let us prove claim (i).  
Let $p_c \colon G \twoheadrightarrow G_c$ denote the canonical homomorphism. 
Since  $|K|=8$ (by Corollary \ref{cor:Kclassification}) and $|\ker (p_c)|=3^b$  are coprime, the group $K$ is isomorphic to the image $K_c := p_c(K)$, i.e., 
\[
K_c \simeq Q_8. 
\]
Also, the homomorphism $\pi_c \colon G_c \twoheadrightarrow A_4$ maps the group $K_c$ onto the group $\overline{K} = V_4$, i.e., 
\[
\pi_c(K_c) = V_4.
\]
If the central extension 
\[
1 \to \ZZ/c\ZZ \to G_c \stackrel{\pi_c}{\to} A_4 \to 1
\]
splits, then $G_c \simeq \ZZ/c\ZZ \times A_4$, and hence the group $K_c$ embeds into $\ZZ/c\ZZ \times V_4$. 
This contradicts the fact that $K_c \simeq Q_8$ is non-abelian.
Thus, the central extension
    \[
    1\to \ZZ/c\ZZ\to G_c\to A_4\to 
    1\]
    dose not split, and it defines the unique nontrivial element in  $H^2(A_4,\ZZ/c\ZZ)$. 
Hence the  assertion follows from Lemma \ref{lem:str_of_ext_H^2_A_4_case}.


Next, let us prove claim (ii).        
Put $p_3 \colon G \twoheadrightarrow G_3$ and $\pi_3 \colon G_3 \twoheadrightarrow A_4$. 
Since $|K|=8$ and $|\ker(\pi_3)|=3^ b$ are coprime, the homomorphism $\pi_3$ maps the group $p_3(K)$ isomorphically onto the group $\pi(K)= V_4$, i.e., 
\[
\pi_3 \colon p_3(K) \stackrel{\sim}{\to} V_4. 
\]
Thus we obtain the following commutative diagram with exact rows: 
\[
\xymatrix{
1 \ar[r] & \ZZ/3^b\ZZ \ar[r] \ar@{=}[d] 
  & G_3 \ar[r]^{\pi_3} \ar@{->>}[d] 
  & A_4 \ar[r] \ar@{->>}[d] 
  & 1 \\
1 \ar[r] & \ZZ/3^b\ZZ \ar[r] 
  & G_3/p_3(K) \ar[r] 
  & A_4^{\mathrm{ab}} \ar[r] 
  & 1. 
}
\]
This is a pullback diagram, that is, 
\[
G_3 \simeq  A_4  \times_{A_4^{\mathrm{ab}}} G_3/K_3. 
\]
Since $G_3/p_3(K)$ is a quotient of the cyclic group $G/K \stackrel{\sim}{\to} \bZ/d\bZ$, it is itself cyclic.
Therefore, $G_3/p_3(K) \simeq \bZ/3^{b+1}\bZ$, which completes the proof. 
\end{proof}


We now determine the structure of $G$ in the case $(\overline{G}, \overline{K}) = (A_4, V_4)$.

\begin{prop}
\label{prop:A4_type_image}
Suppose that $(\overline{G}, \overline{K}) = (A_4, V_4)$. 
Then the group $G$ can be expressed as
\[
G \simeq (\SL_2(\FF_3) \times_{\ZZ/3\ZZ} \ZZ/2d\ZZ)/\langle (-I, d) \rangle,
\]
where the fiber product is taken over an arbitrary surjection $\SL_2(\FF_3) \twoheadrightarrow \ZZ/3\ZZ$ and the natural projection $\ZZ/2d\ZZ \twoheadrightarrow \ZZ/3\ZZ$. 
\end{prop}
\begin{proof}
Since $3^{b+1}c = 3a = 2d$, it follows Lemma \ref{lem:central_extension_A_4_V_4} together with the isomorphism \eqref{eq:A_4_case_isom_G_3_and_G_c} that 
\begin{align*}
G &\simeq G_c \times_{A_4} G_3
\\
&\simeq G_c \times_{\bZ/3\bZ} \ZZ/3^{b+1}\ZZ
\\
&\simeq  (\SL_2(\bF_3) \times \bZ/c\bZ)/\langle (-I, d) \rangle \times_{\bZ/3\bZ} \ZZ/3^{b+1}\ZZ
\\
&\simeq (\SL_2(\bF_3) \times_{\bZ/3\bZ} \bZ/2d\bZ)/\langle (-I, d) \rangle. 
\end{align*}
\end{proof}
\subsubsection{$S_4$-case.}
We now assume that $(\overline{G},\overline{K})=(S_4,A_4)$. 
In this case, it follows from Corollary \ref{cor:order2} that 
\[
|\overline{G}/\overline{K}|=2,  \quad a=d. 
\]
Since $S_4^{\mathrm{ab}}\simeq \ZZ/2\ZZ$ and $H_2(S_4,\ZZ)\simeq \ZZ/2\ZZ$ (see \cite[(2.21), Chapter 3]{Suzuki82}, for example), the exact sequence given in Proposition \ref{prop:UCT} with $H = S_4$ and $A=\ZZ/d\ZZ$ yields a split exact sequence
 \[
 0\to \operatorname{Ext}^1_{\bZ}(\ZZ/2\ZZ, \ZZ/d\ZZ)\to H^2(S_4,\ZZ/d\ZZ)\to \operatorname{Hom}(\ZZ/2\ZZ,\ZZ/d\ZZ)\to 0.
 \]
Since $d$ is an even integer, we have an isomorphism 
\[
H^2(S_4,\ZZ/d\ZZ) \simeq \ZZ/2\ZZ \times \ZZ/2\ZZ.
\]

\begin{rem}\label{rem:four_element_in_H^2(S_4,Z/2Z)}
The following four groups are pairwise non-isomorphic:
\[
S_4 \times \bZ/2\bZ, \quad
S_4 \times_{\bZ/2\bZ} \bZ/4\bZ, \quad
BO_{48}, \quad
\GL_2(\bF_3).
\]
For instance, this can be seen easily by Remark \ref{rem:double_cover_S_4} and by comparing their abelianizations.
It therefore follows that the four elements of
$H^2(S_4,\ZZ/2\ZZ) \simeq \bZ/2\bZ \times \bZ/2\bZ$
are represented by the central extensions associated with these four groups.
For each of the four groups $\mathcal{G}$ above, we let $[\mathcal{G}] \in H^2(S_4, \bZ/2\bZ)$ denote the cohomology class represented by the central extension associated with $\mathcal{G}$.
Note that 
\[
\operatorname{Ext}^1_{\bZ}(\ZZ/2\ZZ, \ZZ/2\ZZ) = \{[S_4 \times \bZ/2\bZ], \, [S_4 \times_{\bZ/2\bZ} \bZ/4\bZ]\}. 
\]
\end{rem}

We denote by $[S_4 \times_{\bZ/2\bZ} \bZ/2d\bZ]$ the nontrivial element of
$\operatorname{Ext}^1_{\ZZ}(\ZZ/2\ZZ, \ZZ/d\ZZ) \subset H^2(S_4, \bZ/d\bZ)$
determined by the central extension
\[
1 \to \bZ/d\bZ \to S_4 \times_{\bZ/2\bZ} \bZ/2d\bZ \to S_4 \to 1,
\]
and by $[G]$ the element of $H^2(S_4, \bZ/d\bZ)$ determined by the central extension
\eqref{eq:central_extension_Z/a_G_Gbar}.

We consider the homomorphism
\[
j_* \colon H^2(S_4,\ZZ/2\ZZ) \to H^2(S_4,\ZZ/d\ZZ)
\]
induced by the injective homomorphism $j \colon \ZZ/2\ZZ \hookrightarrow \ZZ/d\ZZ$.


\begin{lem}\label{lem:S_4_H_2_genrator}
The cohomology group $H^2(S_4, \mathbb{Z}/d\mathbb{Z})$ is generated by the classes
$[S_4 \times_{\mathbb{Z}/2\mathbb{Z}} \mathbb{Z}/2d\mathbb{Z}]$ and $j_*[BO_{48}]$. 
\end{lem}
\begin{proof}
The injective homomorphism $j \colon \ZZ/2\ZZ \hookrightarrow \ZZ/d\ZZ$ 
induces a commutative diagram with exact rows:
\begin{align}\label{eq:comutative_diagram_S_4_H^2_UCT_Z/2Z_Z/dZ}
\begin{split}
    \xymatrix{
0 \ar[r] & \operatorname{Ext}^1_{\ZZ}(\ZZ/2\ZZ, \ZZ/2\ZZ) \ar[d] \ar[r] 
  & H^2(S_4, \ZZ/2\ZZ) \ar[r] \ar[d]^-{j_*} 
  & \operatorname{Hom}(\ZZ/2\ZZ, \ZZ/2\ZZ) \ar[r] \ar[d]^-{\simeq} 
  & 0
  \\
0 \ar[r] & \operatorname{Ext}^1_{\ZZ}(\ZZ/2\ZZ, \ZZ/d\ZZ) \ar[r] 
  & H^2(S_4, \ZZ/d\ZZ) \ar[r] 
  & \operatorname{Hom}(\ZZ/2\ZZ, \ZZ/d\ZZ) \ar[r] 
  & 0. 
}    
\end{split}
\end{align}
It follows from Remark~\ref{rem:four_element_in_H^2(S_4,Z/2Z)} that
\[
j_*[BO_{48}] \notin \operatorname{Ext}^1_{\ZZ}(\ZZ/2\ZZ, \ZZ/d\ZZ).
\]
Since the (image of the) subgroup
$\operatorname{Ext}^1_{\ZZ}(\ZZ/2\ZZ, \ZZ/d\ZZ)$
is generated by the class
$[S_4 \times_{\ZZ/2\ZZ} \ZZ/2d\ZZ]$,
the commutative diagram
\eqref{eq:comutative_diagram_S_4_H^2_UCT_Z/2Z_Z/dZ}
shows that $H^2(S_4, \ZZ/2d\ZZ) \simeq \ZZ/2\ZZ \times \ZZ/2\ZZ$
is generated by the classes $[S_4 \times_{\ZZ/2\ZZ} \ZZ/2d\ZZ]$ and $j_*[BO_{48}]$. 
\end{proof}

\begin{rem}\label{rem:S_4_j*_d/2}
The commutative diagram \eqref{eq:comutative_diagram_S_4_H^2_UCT_Z/2Z_Z/dZ} shows the following:
\begin{itemize}
  \item[(i)] If $d/2$ is odd, then the homomorphism $j_*$ is an isomorphism.
  \item[(ii)] If $d/2$ is even, then $j_*[BO_{48}] = j_*[\GL_2(\bF_3)]$.
\end{itemize} 
\end{rem}

\begin{lem}\label{lem:S_4_G_is_not_isomorphic_to_BO_48}
$[G] \neq j_*[BO_{48}]$. 
\end{lem}
\begin{proof}
Let $z \in BO_{48}$ denotes the unique element of order $2$ (see Remark \ref{rem:double_cover_S_4}). 
Suppose, to the contrary, that $[G] = j_*[BO_{48}]$, which implies that
\[
G \simeq (BO_{48} \times \bZ/d\bZ)/\langle (z, d/2) \rangle.
\] 
Note that, in this case, the center of $G$ is $\bZ/d\bZ$ and the subgroup $K \subset G$ corresponds to the subgroup $\SL_2(\bF_3) \subset BO_{48}$ by Lemma \ref{cor:Kclassification}. 
Let
\[
L := \overline{\mathbb{Q}}^{\ker(\rho_f)}
\]
be the fixed field corresponding to the open subgroup $\ker(\rho_f)$.
Then the homomorphism $\rho_f$ induces an isomorphism
\[
\operatorname{Gal}(L/\QQ) \simeq G.
\]
Let $G' \subset G$ be the subgroup corresponding to $BO_{48}$.
Set $M := L^{K}$ and $M' := L^{G'}$.
Since $K \subset G'$ and $G'/K \simeq \bZ/2\bZ$, we have
\[
[M : M'] = 2.
\]
Moreover, since $\operatorname{Gal}(M/\QQ) \simeq G/K \simeq \bZ/d\bZ$, 
the extension $M/\QQ$ is cyclic.
Hence $M'/\QQ$ is the unique subextension of $M/\QQ$ of index $2$,
and in particular $M'$ is totally real.
Since $L$ is totally imaginary as $\rho_f$ is odd, it follows that
for every complex conjugation $c \in G_\bQ$, the element $\rho_f(c)$
belongs to $G' \simeq BO_{48}$, i.e., 
\[
\rho_f(c) \in G' \simeq BO_{48}.
\]
On the other hand, since $\rho_f$ is odd, the element $\rho_f(c) \in \GL_2(\bC)$ has
eigenvalues $1$ and $-1$.
In particular, $\rho_f(c)$ does not lie in the center of $G'$.
This contradicts the fact that the center of \(BO_{48}\) is the only subgroup of order 2
(see Remark~\ref{rem:double_cover_S_4}). 
\end{proof}

\begin{prop}\label{prop:[G]_is_sum_of_[BO_48]_and_[S4Z/2dZ]}
$[G] = j_*[BO_{48}] +  [S_4 \times_{\bZ/2\bZ} \bZ/2d\bZ]$. 
\end{prop}
\begin{proof}
Since $\SL_2(\FF_3) \simeq K \subset G$ by Lemma~\ref{cor:Kclassification}, the group $G$ cannot be isomorphic to either
$S_4 \times \bZ/d\bZ$ or $S_4 \times_{\bZ/2\bZ} \bZ/2d\bZ$.
In particular, $[G] \neq 0$ and
$[G] \neq [S_4 \times_{\bZ/2\bZ} \bZ/2d\bZ]$.
Since $H^2(S_4, \bZ/d\bZ) \simeq \bZ/2\bZ \times \bZ/2\bZ$, it follows from Lemmas \ref{lem:S_4_H_2_genrator} and \ref{lem:S_4_G_is_not_isomorphic_to_BO_48}
that $[G] = j_*[BO_{48}] + [S_4 \times_{\bZ/2\bZ} \bZ/2d\bZ]$. 
\end{proof}

\begin{prop}
\label{prop:S4_type_image}
Suppose that $(\overline{G},\overline{K})=(S_4,A_4)$
Then the group $G$ is isomorphic to the group 
    \[
(BO_{48} \times_{\bZ/2\bZ} \ZZ/2d\ZZ)/\langle (z, d) \rangle.
    \]
    Here, $z \in BO_{48}$ denotes the unique element of order $2$ and the homomorphism $BO_{48} \twoheadrightarrow \ZZ/2\ZZ$ used in the fiber product is defined as the composition
\[
BO_{48} \twoheadrightarrow BO_{48}^{\rm ab} \simeq  \ZZ/2\ZZ.
\]
\end{prop}






\begin{proof}
This proposition follows from Proposition~\ref{prop:[G]_is_sum_of_[BO_48]_and_[S4Z/2dZ]}
(note that the sum of central extensions can be computed using the Baer sum;
see \cite[Definition~3.4.4]{Weibel94} for the definition of the Baer sum). 
\end{proof}

\begin{rem}\label{rem:Glois_image_GL(2,3)_expression_S_4_K_neq_G_case}
The group $G$  admits the following description in terms of $\GL_2(\bF_3)$. 
\begin{itemize}
    \item[(i)]    If $d/2$ is odd, then Proposition~\ref{prop:[G]_is_sum_of_[BO_48]_and_[S4Z/2dZ]}, together with Remark \ref{rem:S_4_j*_d/2}(i), implies that 
\begin{align*}
[G] &= j_*[BO_{48}] + [S_4 \times_{\bZ/2\bZ} \bZ/2d\bZ]  \\
    &= j_*([BO_{48}] + [S_4 \times_{\bZ/2\bZ} \bZ/4\bZ]) \\
    &= j_*[\GL_2(\bF_3)].
\end{align*}
It follows that
\[
G \simeq \GL_2(\bF_3) \times \bZ/(d/2)\bZ.
\]
    \item[(ii)]  If $d/2$ is even, then Proposition~\ref{prop:[G]_is_sum_of_[BO_48]_and_[S4Z/2dZ]}, together with Remark \ref{rem:S_4_j*_d/2}(ii), implies that 
\begin{align*}
[G] = j_*[\GL_2(\bF_3)] + [S_4 \times_{\bZ/2\bZ} \bZ/2d\bZ].
\end{align*}
It follows that
\[
G \simeq (\GL_2(\bF_3) \times_{\bZ/2\bZ} \ZZ/2d\ZZ)/\langle (-I, d) \rangle.
\]
\end{itemize}
\end{rem}

\bibliography{References.bib}
\bibliographystyle{amsplain}

\end{document}